\documentclass[a4paper,11pt,reqno]{amsart}
\usepackage{amscd,graphicx,amssymb,color}

\def\M{\overline{M}}
\def\Ga{\Gamma}
\def\De{\Delta}
\def\h{\text{h}}
\def\H{\text{H}}
\def\dra{\dashrightarrow}
\def\hra{\hookrightarrow}

\def\ra{\rightarrow}

\def\T{\tilde T}

\newcommand{\bp}{\begin{proofx}}
\newcommand{\ep}{\end{proof}}  

\def\bi{\begin{itemize}}
\def\ei{\end{itemize}}  

\def\uone{\underline1}
\def\uzero{\underline0}

\def\oM{\overline{M}}

\def\bP{\Bbb P}

\def\SS{\tilde S}

\def\Hes{\operatorname{Hes}}
\def\Ceva{\operatorname{Ceva}}

\def\bD{\Bbb D}

\def\Aut{\operatorname{Aut}}

\def\Gr{\operatorname{Gr}}
\def\de{\delta}

\def\cO{\Cal O}

\def\codim{\operatorname{codim}}

\def\bZ{\Bbb Z}

\def\bQ{\Bbb Q}
\def\bG{\Bbb G}
\def\bR{\Bbb R}
\def\bA{\Bbb A}

\def\bX{\Bbb X}

\def\cD{\Cal D}
\def\cS{\Cal S}

\def\tS{\tilde S}

\def\oS{\overline{S}}

\def\tS{\tilde{S}}

\def\cW{\Cal W}

\def\oN{\overline{N}}

\def\cD{\Cal D}

\def\cH{\Cal H}
\def\cM{\Cal M}
\def\cZ{\Cal Z}

\def\Spec{\operatorname{Spec}}

\def\PGL{\operatorname{PGL}}
\def\HH{\operatorname{H}}

\def\ga{\gamma}

\def\cO{\Cal O}

\def\O{\cO}

\def\Pic{\operatorname{Pic}}

\def\oPic{\overline{\Pic}}
\def\oPicone{\oPic^{\uone}}
\def\ocPic{\overline{\cPic}}

\def\Spec{\operatorname{Spec}}

\def\Hom{\operatorname{Hom}}
\def\Cal{\mathcal}

\def\bX{\Bbb X}

\def\ov{\overline}

\def\eps{\varepsilon}

\def\al{\alpha}
\def\be{\beta}
\def\ga{\gamma}

\def\arrow{\mathop{\longrightarrow}\limits}

\def\arrow{\mathop{\longrightarrow}\limits}

\def\oM{\overline{M}}
\def\os{\overline{S}}

\def\ocM{\overline{\Cal M}}

\def\tN{\tilde{N}}
\def\tS{\tilde{S}}
\def\ts{\tS}
\def\ts7{\tilde{S}_7}

\def\bP{\Bbb P}

\def\cC{\Cal C}

\def\Bl{\operatorname{Bl}}

\def\Aut{\operatorname{Aut}}

\def\Gr{\operatorname{Gr}}

\def\Hom{\operatorname{Hom}}

\def\cO{\Cal O}
\def\cS{\Cal S}
\def\cD{\Cal D}

\def\Pic{\operatorname{Pic}}

\def\codim{\operatorname{codim}}

\def\bZ{\Bbb Z}

\def\bQ{\Bbb Q}
\def\bG{\Bbb G}
\def\bR{\Bbb R}
\def\bA{\Bbb A}
\def\PP{\Bbb P}
\def\dra{\dashrightarrow}

\def\bX{\Bbb X}

\def\cD{\Cal D}
\def\cS{\Cal S}

\def\tS{\tilde S}

\def\ep{\ddot p}

\def\oS{\overline{S}}

\def\tS{\tilde{S}}

\def\cW{\Cal W}

\def\oN{\overline{N}}

\def\cH{\Cal H}
\def\cM{\Cal M}

\def\Spec{\operatorname{Spec}}

\def\PGL{\operatorname{PGL}}
\def\PGL{\operatorname{PGL}}

\def\Exc{\operatorname{Exc}}
\def\Eff{\operatorname{Eff}}

\def\Cal{\mathcal}

\def\Pic{\operatorname{Pic}}

\def\oN{\overline{N}}
\def\os7p{\oS_7'}
\def\os7{\oS_7}
\def\on6{\oN_6}
\def\n6{\oN_6}

\def\bSi{\bold\Sigma}

\def\Mor{\operatorname{Mor}}

\def\cG{\Cal G}

\def\eset{\emptyset}
\def\cPic{\Cal P ic}

\def\oNE{\ov{\operatorname{NE}}}
\def\oEff{\ov{\operatorname{Eff}}}
\def\NS{\operatorname{NS}}

\newtheoremstyle{mystyle}{}{}{\itshape}{}{\scshape}{.}{ }{}
\theoremstyle{mystyle}

\swapnumbers
\newtheorem{Theorem}{Theorem}[section]

\newtheorem{Proposition}[Theorem]{Proposition}
\newtheorem{Lemma}[Theorem]{Lemma}
\newtheorem{lemdef}[Theorem]{Lemma--Definition}
\newtheorem{Corollary}[Theorem]{Corollary}
\newtheorem{Claim}[Theorem]{Claim}

\newtheoremstyle{myreview}{}{}{}{}{\scshape}{.}{ }{}
\theoremstyle{myreview}

\newtheorem{Definition}[Theorem]{Definition}
\newtheorem{Example}[Theorem]{Example}

\newtheorem{Remark}[Theorem]{Remark}

\newtheorem{Notation}[Theorem]{Notation}
\newtheorem{Review}[Theorem]{}

\newcounter{et}[Theorem]
\def\cooltag{\tag{\arabic{section}.\arabic{Theorem}.\arabic{et}}\addtocounter{et}{1}}

\begin{document}

\title[Exceptional Loci on $\oM_{0,n}$ and Hypergraph Curves]{Exceptional Loci on $\oM_{0,n}$\\ and Hypergraph Curves}
\author{Ana-Maria Castravet and Jenia Tevelev}

\begin{abstract}
We give a myriad of examples of extremal divisors, rigid curves, and birational morphisms
with unexpected properties for the Grothendieck--Knudsen moduli space $\oM_{0,n}$ 
of stable rational curves.
The basic tool is an isomorphism between $M_{0,n}$ and the Brill--Noether
locus of a very special reducible curve corresponding to a hypergraph.
\end{abstract}

\maketitle

\tableofcontents

\section{Introduction}

For any projective variety $X$, the basic gadgets encoding
combinatorics of 
its birational geometry are 
the Mori cone $\oNE_1\subset\NS_\bR^*$ (the closure of 
the cone generated by effective $1$-cycles) and 
the effective cone $\ov\Eff\subset \NS_\bR$ 
(the closure of the cone generated by effective Cartier divisors).
Here $\NS$ is the Neron-Severi group of $X$.

Let  $X=\oM_{0,n}$ be the moduli space of stable rational 
curves with $n$ marked points.
It is stratified by the topological type of a stable rational curve
and so it has ``natural'' boundary effective divisors and curves.
For example, $\oM_{0,5}$ is isomorphic to the blow-up
of $\bP^2$ in $4$ points, and boundary divisors are the ten
 $(-1)$-curves.
They generate $\oEff(\oM_{0,5})=\oNE_1(\oM_{0,5})$. 
For $n=6$, 
Keel (unpublished) and Vermeire~\cite{V} showed that $\oEff(\oM_{0,6})$ 
is not generated by classes of boundary divisors.
A new divisor has many interesting geometric interpretations
but the known proof of its extremality is ``numerical''
rather than geometric in flavor, namely 
one computes its class and shows that it can not be written as a 
nontrivial sum of pseudoeffective divisors. 
To the best of our knowledge, this was the only known extremal divisor on $\oM_{0,n}$
different from boundary divisors (of course one can also take pull-backs
of the Keel--Vermeire divisor for various forgetful maps $\oM_{0,n}\to\oM_{0,6}$).

Keel and McKernan 
\cite{KM} proved  that $\oNE_1(\oM_{0,6})$ 
is generated by classes of boundary curves 
(according to \cite{KM}, Fulton conjectured that this holds for any~$n$).
They also proved that 
a hypothetical counterexample to Fulton's conjecture  must be a rigid curve.
More precisely,  assume, for simplicity, that $\oNE_1(\oM_{0,n})$ is finitely generated
(otherwise the Fulton's conjecture is definitely false).
Suppose also that a curve $C$ generates an extremal ray.
Then if 
$C\cap M_{0,n}\ne\emptyset$ then $C$ must be a rigid curve.
This statement is not explicitly stated in \cite{KM} in this form,
but can be immediately proved using their methods,
so we included it in the Appendix for the reader's convenience.
Thus it is natural to ask if $\oM_{0,n}$ has rigid curves intersecting~$M_{0,n}$.

Our approach is to study exceptional loci of 
birational morphisms $\oM_{0,n}\to Z$ defined by hypergraphs.
A hypergraph 
$$\Gamma=\{\Gamma_1,\ldots,\Gamma_d\}$$ 
is a collection of subsets (called hyperedges) of the set $N:=\{1,\ldots,n\}$.
We define the hypergraph morphism as the product of forgetful morphisms
$$\pi_\Gamma=\pi_{\Gamma_1}\times\ldots\times\pi_{\Gamma_d}:\,\oM_{0,n}\to
\oM_{0,\Gamma_1}\times\ldots\times\oM_{0,\Gamma_d},$$
where
$\pi_I:\,\oM_{0,n}\to\oM_{0,I}$
is the morphism given by dropping the points of $N\setminus I$
(and stabilizing). Here $Z=\pi_\Gamma(\oM_{0,n})$.
Of course this is the most obvious morphism to consider, but
a priori it is not clear at all how to study~it.
Our main innovation is to interpret $\pi_\Gamma$ in terms 
of Brill--Noether loci of a (reducible) {\em hypergraph curve} (see Section~\ref{BNSection}).
This allows us to study its exceptional loci explicitly.
It would take too long to reproduce all the relevant results in the Introduction,
so, to help the reader, let us give the following (very close) analogy,
which we find useful. Suppose $C$ is a general smooth curve of genus $g$.
By the Brill--Noether theory~\cite{ACGH}, $W^1_{g+1}=\Pic^{g+1}(C)$, $G^1_{g+1}$ is smooth,
and the morphism $G^1_{g+1}\to W^1_{g+1}$ has an exceptional divisor $E$ 
contracted to the locus $W^2_{g+1}$ (of codimension~$3$ in $\Pic^{g+1}$).
So we see immediately that $E$ generates an edge of $\oEff(G^1_{g+1})$. 
This is exactly what we do in this paper to produce exceptional loci on $\oM_{0,n}$,
except that we take a very reducible hypergraph curve instead of a smooth curve
(with some caveats, see below).

The fact that exceptional loci of hypergraph morphisms intersect~$M_{0,n}$
is an amusing feature of genus zero: if~$g\ge1$ then the exceptional locus
of any birational morphism $\oM_{g,n}\to Z$ belongs to the boundary $\oM_{g,n}\setminus M_{g,n}$,
see \cite[0.11]{GKM}. The main result of \cite{GKM} is the reduction of the Fulton's
conjecture for $\oM_{g,m}$ to the Fulton's conjecture for $\oM_{0,n}$.
However, 
our results show that one has to exercise caution because the birational
geometry of $\oM_{0,n}$ is different from the higher genus case.
Our exceptional loci can have irreducible components of all
possible dimensions. Our record is a morphism $\oM_{0,12}\to Z$
with a one-dimensional component of the exceptional locus intersecting $M_{0,12}$
(see Theorem~\ref{HesseMain}) that we obtain using the dual Hesse hypergraph.
The exceptional curve is obviously rigid.
However, we don't know if this curve (or any other exceptional curve)
is a counterexample 
to Fulton's conjecture.
Along the way, we also discover an amusing fact (Prop.~\ref{blowupdescr}) that $\oM_{0,n}$
is covered by blow-ups of $\bP^2$ in $n$ points, generalizing the well-known fact
that moduli of cubic surfaces are generated by hyperplane sections of the Segre cubic threefold.

If we arrange $Z$ to be smooth then the exceptional locus 
is divisorial and generates an extremal ray of~$\ov\Eff(\oM_{0,n})$
(see Corollary~\ref{manyrays}).
We devote the bulk of this paper to the study of these divisors.
For example, we show that the number of new extremal rays of this form grows rapidly with $n$
(see Theorem~\ref{NewDivisors}) by giving a ``Fibonaccian'' inductive construction (Theorem~\ref{FunnyConstruction})
for many of them.
We give various geometric descriptions of these divisors, find out when they are irredicible, etc.

Let us mention one interesting feature of our argument.
In the ``divisorial'' set-up, one of the markings is special, and the image~$\cD$ of the extremal divisor 
with respect to the forgetful morphism $p:\,\oM_{0,n}\to\oM_{0,n-1}$ is also a divisor. 
Since $p$ is flat,
$\cD$ must also generate an extremal ray. 
We show that the Keel--Vermeire
divisor in $\oM_{0,6}$ is the first example of such a divisor (see Example~\ref{KeelVermeire}).
This gives a  transparent
geometric proof of its extremality: though 
it can not be contracted by a birational morphism itself,
its preimage in $\oM_{0,7}$ is contracted by a hypergraph morphism $\oM_{0,7}\to(\bP^1)^4$.
We find this trick of proving extremality of divisors by flat base change quite unexpected.

Here is the outline of the paper.
In Section~\ref{BNSection} we introduce hypergraph curves 
and ther Brill-Nother loci. In Section~\ref{DivSection} we give
our main results about hypergraph divisors.
Most of the proofs are given later, in Sections~\ref{projections} and~\ref{divisorssection}.
In Section~\ref{smaps} we study admissible sheaves on hypergraph curves.
Our results about exceptional and rigid curves on $\oM_{0,n}$ are given in \ref{CurvesSection}.
In the Appendix, we rearrange the proof in \cite{KM}
to show that a (hypothetical) counterexample to Fulton's conjecture must be rigid.

\medskip
\noindent\textsc{Acknowledgements.}
We are very grateful to Sean Keel for many enlightening conversations. We would like to thank James M\textsuperscript{c}Kernan
and Sean Keel for the permission to include their result in the Appendix.
We would like to thank  Valery Alexeev, Lucia Caporaso, Igor Dolgachev, Gabi Farkas, Brendan Hassett,
Anna Kazanova, Bernd Sturmfels and Giancarlo Urzua
for useful discussions. Part of this paper was written while the first author was visiting the Max-Planck Institute in Bonn, Germany.
The second author was partially supported by the NSF grant 
DMS-0701191 and by the Sloan research fellowship.

\section{Hypergraph Curves and their Brill--Noether Loci}\label{BNSection}

We work with schemes over an algebraically closed field $k$.
Let $\Gamma=\{\Gamma_1,\ldots,\Gamma_d\}$ be a hypergraph.
For any $i\in N$, let the valence  $v_i$  
be the number of hyperedges containing $i$.
We assume that
$|\Gamma_i|\ge3$ for any $i$, that $\Gamma$ is connected, and
that $v_i\ge2$ for any $i$.

A curve $\Sigma$ is called a hypergraph curve if it has $d$ irreducible
components, each isomorphic to $\bP^1$ and marked by $\Gamma_i$.
These components are
glued at identical markings as follows:
at each singular point $i\in N$, $\Sigma$
is locally isomorphic to the union of coordinate axes in $\bA^{v_i}$.
A nodal curve $\Sigma^s$, called a stable hypergraph curve, or a stable model of $\Sigma$,
is obtained by inserting a $\bP^1$ with $v_i$ markings 
instead of each singular point of $C$ with $v_i>2$ (see Fig.~\ref{Tetrahedron}).
A stable hypergraph curve is a special case of a graph curve
of Bayer and Eisenbud \cite{BE}, which explains our terminology.

\begin{figure}[htbp]
\includegraphics[width=5in]{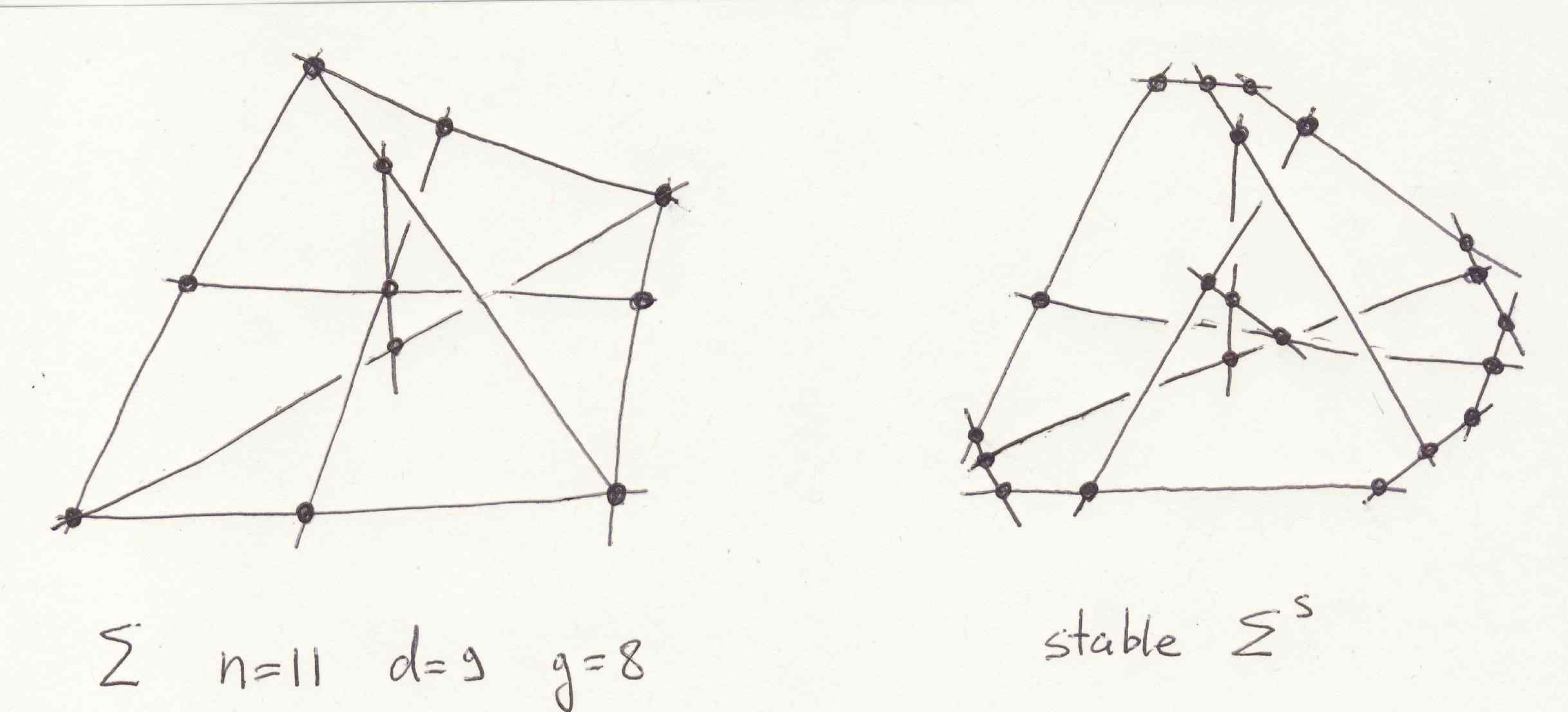}
\caption{\small A hypergraph curve.
}\label{Tetrahedron}
\end{figure}
 
Fix $\Gamma$ and consider 
the functor $\cM_\Gamma:\,Schemes\to Sets$
that associates to a $k$-scheme $S$ the set of isomorphism classes of flat families $\bSi\to S$
such that geometric fibers are  hypergraph  curves.
It  is represented by 
$M_\Gamma:=\prod\limits_{j=1\ldots d} M_{0,\Gamma_j}$.
Similarly, the functor $Schemes\to Sets$
of flat families of stable hypergraph curves  
is represented by 
$M_\Gamma\times\prod\limits_{i=1\ldots n\atop v_i>2} M_{0,v_i}$.
Let $g=p_a(\Sigma)=p_a(\Sigma^s)$ be the arithmetic genus.
It can be computed as follows:
$$g=\dim M_\Gamma+2d-n+1.$$

The universal stable hypergraph
curve is a pull-back of the universal curve of $\ocM_g$ but it is probably worth emphasizing that 
stable hypergraph curves correspond to the same hypergraph $\Gamma$, and so they
have the same number of components, etc. They are not isomorphic only because
cross-ratios of points of attachment are different.

Consider the functor  $\cPic^{\underline i}:Schemes\to Sets$ that sends a scheme $S$
to an object $\bSi$ of $\cM_\Gamma(S)$  equipped with 
an invertible sheaf on $\bSi$ that has degree $i$ on each irreducible component
modulo pull-backs of invertible sheaves on~$S$. Then $\cPic^{\underline 0}\simeq \bG_{m}^g\times\cM_\Gamma$ 
(not canonically)
and $\cPic^{\underline i}$ is a (non-canonically trivial) $\cPic^0$-torsor.

We call an effective Cartier divisor (resp. a linear system)
on a hypergraph curve {\em admissible} if it does not contain singular points
(resp.~is globally generated and separates singular points).
An invertible sheaf $L$ is called admissible if a complete linear system $|L|$ is admissible.
We define the Brill--Noether loci 
following the standard notation of~\cite{ACGH}. Their geometric points have the following description:
$$W^r=\{\hbox{\rm a curve  $\Sigma$,  an admissible $L\in\Pic^{\uone}(\Sigma)$ such that}\  h^0(\Sigma,L)\ge r+1\},$$  
$$C^r=\{\hbox{\rm $\Sigma$, an admissible divisor}\ D\ \hbox{\rm on $\Sigma$ such that}\  \cO(D)\in W^r\},$$
$$G^r=\{L\in W^r,\ \hbox{\rm an admissible pencil}\ V\in\Gr(2,H^0(\Sigma,L))\}\quad (r\geq1),$$
$$\tilde G^r=\{(L,V)\in G^r,\ \hbox{\rm an admissible}\ D\in|V|\},$$
see Fig.~\ref{KHyper} for the illustration.
Note that the image of $\Sigma$ under the complete linear system $|L|$ with 
$L\in W^r\setminus W^{r+1}$ is a union of $d$ intersecting lines in $\bP^r$
(some of which could be equal)
with points of $\Gamma_j$ on $j$-th line.
For example, the curve $\Sigma$  in Fig.~\ref{Tetrahedron} is embedded
by $L\in W^3$.

\begin{figure}[htbp]
\includegraphics[width=3in]{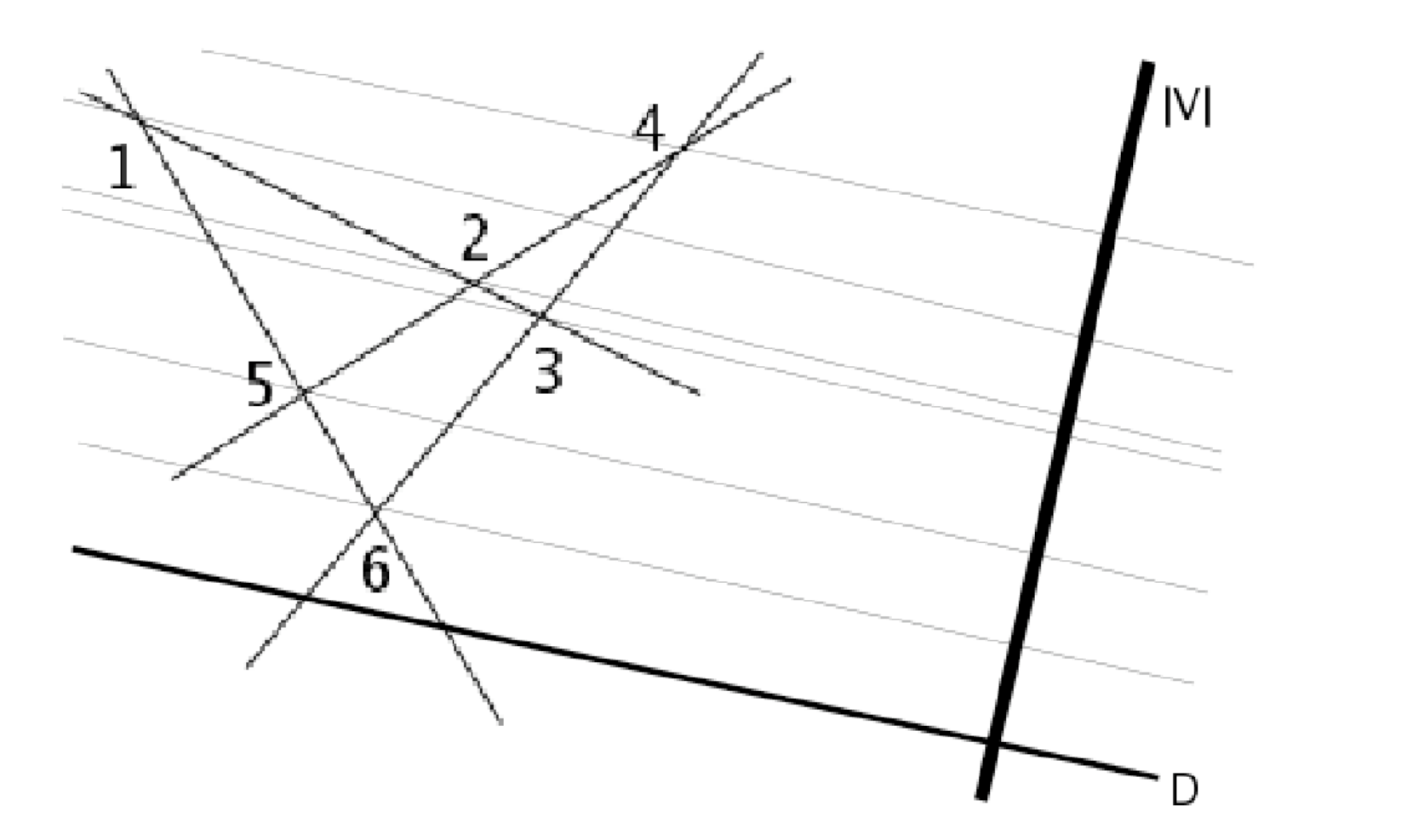}
\caption{\small A curve that corresponds to a hypergraph $\{123,524,346,156\}$ embedded in $\bP^2$ 
by a line bundle $L\in W^2(\Sigma)$, a pencil $|V|$, and a divisor $D\in|V|$.
}\label{KHyper}
\end{figure}

Let us give a more detailed account.
Let $\bSi^{sm}$ be the smooth locus of the universal family $\bSi\to\cM_\Gamma$
with irreducible components $\bSi^{sm}_1,\ldots,\bSi^{sm}_d$.

Let $\cG^1$ be the functor $Schemes\to Sets$ that sends a scheme $S$ to the set of isomorphism classes
of (1) a family $\{p:\,\bSi\to S\}\in\cM_\Gamma(S)$ and (2) a morphism  $f:\,\bSi\to\bP^1_S$
such that (a) images of irreducible components of $\bSi\setminus\bSi^{sm}$ are disjoint
and (b) each irreducible component of $\bSi$ maps isomorphically onto $\bP^1_S$.
Here two morphisms are considered isomorphic if they differ by isomorphisms of $S$-schemes 
both on the source and the target.
Let $v:\,\cG^1\to\cPic^1$ be the natural transformation that sends $(\bSi\to S,f:\,\bSi\to\bP^1_S)$
to $(\bSi\to S,f^*\cO_{\bP^1_S}(1)$.
We will see below in Theorem~\ref{ID} that $\cG^1$ is represented by a reduced scheme $G^1$.
Assume this for now.
For any $r\ge1$, let $G^r\subset G^1$ be a closed subset
(with an induced reduced scheme structure) of points where 
$p_*(f^*\cO_{\bP^1_{G^1}}(1))$ has rank at least $r+1$ (where $(p,f)$ is the universal family).
We define $W^r\subset \Pic^1$ as a scheme-theoretic image of $G^r$.

Let 
$$\cC^0=\bSi^{sm}_1\times_{\cM_\Gamma}\ldots\times_{\cM_\Gamma}\bSi^{sm}_d$$
and let  
$u:\,\cC^0\to\cPic^1$
be the Abel map that sends $(p_1,\ldots,p_d)\in \cC^0(S)$
to $\cO_{\bSi}(p_1+\ldots+p_d)$. 
Geometric fibers of $u$ are open subsets of admissible divisors
in complete linear systems on $\Pic^1(\Sigma_k)$.
Let $\cC^r:=u^{-1}(\cW^r)\subset\cC^0$.

Finally, we define $\tilde\cG^1$ as the functor $Schemes\to Sets$
that sends $S$ to the datum $(\bSi\to S,\ f:\,\bSi\to\bP^1_S)\in\cG^1(S)$ and a section $s:\,S\to \bP^1_S$
disjoint from images of irreducible components of $\bSi\setminus\bSi^{sm}$.
We define $\tilde\cG^r$ as the preimage of $\cG^r$ for the forgetful map $\tilde\cG^r\to\cG^r$.
We also have the natural transformation $\tilde\cG^r\to\cC^0$
that sends $(\bSi\to S,\ f:\,\bSi\to\bP^1_S,s)$ to $f^{-1}(s(S))$.
It factors through $\cC^r$.

\begin{Definition}
For a subset $I\subset N$ with $|I|\ge4$, let 
$\pi_I:\,\oM_{0,n}\to\oM_{0,I}$
be the morphism given by dropping the points of $N\setminus I$
(and stabilizing). 
We define the {\em hypergraph morphism} as the product of forgetful morphisms
$$\pi_\Gamma=\pi_{\Gamma_1}\times\ldots\times\pi_{\Gamma_d}:\,\oM_{0,n}\to
\oM_{0,\Gamma_1}\times\ldots\times\oM_{0,\Gamma_d}.$$
An important  special case is a ``cone'' 
$$\Gamma\cup\{n+1\}:=\{\Gamma_1\cup\{n+1\},\ldots,\Gamma_d\cup\{n+1\}\}$$
and the corresponding hypergraph morphism  $\pi_{\Gamma\cup\{n+1\}}$ of $\oM_{0,n+1}$. 
\end{Definition}

Our philosophy is summarized in the following theorem:

\begin{Theorem}\label{ID}\label{r-DimlFibers}
We have a natural commutative diagram
$$\begin{CD}
\cPic^1 @<u<< \cC^0\hbox to 0pt{$\ \ \hookleftarrow\  \cC^1$} @<<<\tilde \cG^1 @>>>\cG^1 @>v>>\cW^1\\
@VVV                         @|                            @|                          @|      @VVV\\
\cM_\Gamma @<{\prod\pi_{\Gamma_j}}<< \prod\limits_{j=1}^d\cM_{0,\Gamma_j\cup\{n+1\}} @<{\pi_{\Gamma\cup\{n+1\}}}<<    
\cM_{0,n+1} @>\pi_N>> \cM_{0,n} @>{\pi_\Gamma}>> \cM_\Gamma                
\end{CD}$$
An~isomorphism $\tilde G^1\simeq M_{0,n+1}$ on geometric points 
takes $(\Sigma,L,V,D)$ to the rational curve $|V|\simeq\bP^1$ with $n+1$
marked points given by $D$ and the images of points $\Sigma^{sing}$.

The set-theoretic union of the fibers of $\pi_{\Gamma\cup\{n+1\}}$ of dimension at least $r-1$ is~$\tilde G^r$.
In particular, the~exceptional locus\footnote{The exceptional locus of a 
morphism is a (set-theoretical) union of positive-dimensional irreducible components of fibers.} 
of $\pi_{\Gamma\cup\{n+1\}}$ is~$\tilde G^2$.
The~exceptional locus 
of $v$ is~$G^2$, and the~exceptional locus
of~of $\pi_{\Gamma}$ contains $G^2$.
\end{Theorem}

\begin{proof}
Note that each datum $(\bSi\to S,\ f:\,\bSi\to\bP^1_S)\in\cG^1(S)$ gives rise
to an isomorphism class of a flat family over $S$ with reduced geometric fibers
given by $\bP^1$ and with $n$ disjoint sections (images of irreducible components of $\bSi\setminus\bSi^{sm}$).
This gives a natural transformation $\cG^1\to\cM_{0,n}$ which is in fact an natural isomorphism, because, given a flat
family of marked $\bP^1$'s, we can just glue $d$ copies of $\bP^1_S$ along sections in each $\Gamma_i$.
More precisely, locally along a section $i\in N$, $\bSi$ is isomorphic to a closed subscheme in $\prod_{i\in\Gamma_j}\bP^1_S$
in an obvious way.
This gives a flat family of hypergraph curves over $S$ and its map to $\bP^1_S$, i.e.,~a datum in $\cG^1(S)$.
The same argument shows that $\tilde\cG^1$ is isomorphic to $\cM_{0,n+1}$
and that $\cC^0$ is isomorphic to $\prod\limits_{j=1}^d\cM_{0,\Gamma_j\cup\{n+1\}}$.
\end{proof}

This theorem introduces three very interesting morphisms:
$\pi_{\Gamma\cup\{n+1\}}$, $\pi_{\Gamma}$, and $v$.
Their exceptional loci will be studied in detail in subsequent sections. 

\section{Extremal Divisors on \rm $\oM_{0,n}$}\label{DivSection}

It is easy to work out  when $\pi_{\Gamma\cup\{n+1\}}$ is birational:

\begin{Theorem}\label{coolcondition}
$\pi_{\Gamma\cup\{n+1\}}$ 
is dominant if and only if 
\begin{equation}\label{CondS}
|\bigcup_{j\in S}\Gamma_j|-2\ge\sum_{j\in S}(|\Gamma_j|-2)
\quad\hbox{\rm for any $S\subset\{1,\ldots,d\}$}.
\tag{\ddag}\end{equation}
It is birational if and only if it is dominant  and dimensions match, i.e.
\begin{equation}\label{matchdim}
n-2=\sum\limits_{j=1}^d(|\Gamma_j|-2).
\cooltag\end{equation}
\end{Theorem}

See \ref{projections} for the proof.
Now we use the following simple observation:

\begin{Lemma}
Consider the diagram of morphisms of projective $\bQ$-factorial varieties
$$\begin{CD}
X @>f>> Y\\
@VpVV\\
Z
\end{CD}$$
Suppose that $f$ is birational and that $p$ is faithfully flat.
Let $D$  be an irreducible  component of the exceptional locus of~$f$.
Then $D$ is a divisor that generates an extremal ray of~$\oEff(X)$. If $p(D)\ne Z$ 
and a generic fiber of $p$ along $p(D)$ is irreducible 
then $p(D)$ is a divisor that generates an extremal ray of~$\oEff(Z)$.
\end{Lemma}

\begin{proof}
$D$ is a divisor by van der W\"arden's purity theorem, see \cite[21.12.12]{EGA4}.
It is easy to show (and well-known) that it generates a ray (in fact an edge, see Def.~\ref{ConvexEdge}) of $\oEff(X)$.
Since $p$ is flat and $p(D)\ne Z$, $p(D)$ is an irreducible
divisor. Since $p^{-1}(p(D))$ is irreducible (e.g.~by \cite[Lem.~2.6]{T}), $p^{-1}(p(D))=D$.
It follows that $p(D)$ generates a ray (in fact an edge) of $\oEff(Z)$ because 
$\oEff(Z)$ injects in  $\oEff(X)$ by the pull-back $p^*$.
\end{proof}

\begin{Corollary}\label{manyrays}
Suppose that we have \eqref{CondS}, \eqref{matchdim}, and that $W^2\ne\emptyset$.
Irreducible components of the exceptional locus~of 
\begin{equation}\label{pgn}
\pi_{\Gamma\cup\{n+1\}}:\,\oM_{0,n+1}\to\prod\limits_{j=1}^d\oM_{0,\Gamma_j\cup\{n+1\}}
\cooltag\end{equation}
that intersect $M_{0,n+1}$ exist, are divisorial,
and generate extremal rays of $\oEff(\oM_{0,n+1})$.
Their images in $\oM_{0,n}$ (for the forgetful morphism $\pi_N$) 
are divisors that generate extremal rays of $\oEff(\oM_{0,n})$.
\end{Corollary}

\begin{Example}\label{KeelVermeire}
Consider the hypergraph $\Gamma$ from Fig.~\ref{KHyper}. The condition of Th.~\ref{coolcondition}
is clearly satisfied and $W^2\ne\emptyset$. Therefore, we get an extremal divisor
on $\oM_{0,7}$ contracted by a birational morphism $\oM_{0,7}\to(\bP^1)^4$
and its image is an extremal divisor on $\oM_{0,6}$.
One can check that this is the Keel--Vermeire divisor
(or apply the main result of \cite{HT} that  Keel--Vermeire divisors
are unique extremal divisors of $\oM_{0,6}$ intersecting the interior).
Because of this, we will call the reducible curve from Fig.~\ref{KHyper}, the {\em Keel--Vermeire curve}.
\end{Example}

One can wonder if all our divisors are pullbacks of extremal divisors on $\oM_{0,k}$
with respect to  forgetful morphisms $\oM_{0,n}\to\oM_{0,k}$
or even pullbacks of the Keel--Vermeire divisor! In the following theorem we show that this is not the case. 

\begin{Theorem}\label{NewDivisors}
Suppose $\Gamma=\{\Gamma_1,\ldots,\Gamma_{n-2}\}$ is a $3$-graph that satisfies
\begin{equation}\label{CondM}
|\bigcup_{i\in S}\Gamma_i|\ge|S|+3
\quad\hbox{\rm for any $S\subset\{1,\ldots,n-2\}$ with $2\le|S|\le n-3$}
\tag{\dag}\end{equation}
and such that $W^2(\Sigma)\ne\emptyset$. There exists an index, say $n$, that appears in 
only two triples, say $\Gamma_{n-3}$ and  $\Gamma_{n-2}$. Let 
$$\Gamma'=\{\Gamma_1,\ldots,\Gamma_{n-4}\}$$ 
be a hypergraph 
on $N'=N\setminus\{n\}$ and let $\Sigma'$ be the corresponding curve. Assume that 
$W^3(\Sigma')=\emptyset$.

Then we have the following: the exceptional divisor of $\pi_{\Gamma\cup\{n+1\}}$ contains an
irreducible component intersecting $M_{0,n+1}$ such that
its image $D$ in $\M_{0,n}$ surjects onto $\oM_{0,n-1}$
for any forgetful map $\pi_{N\setminus\{i\}}$, $i=1,\ldots,n$.
In particular, $D$ is not a pull-back from $\oM_{0,k}$ for $k<n$.
\end{Theorem}

\begin{Remark}
Condition ($\dag$) implies that each index appears in two or three triples, with exactly $6$ indices appearing only in two triples.
\end{Remark}

We give sufficient conditions for irreduciblity of the exceptional locus of $\pi_{\Gamma\cup\{n+1\}}$  on $M_{0,n+1}$: 
\begin{Proposition}\label{Irred}
In the set-up of Theorem \ref{NewDivisors}, assume in addition that 
the locus in $W^2(\Sigma')$ whose geometric points $(\Sigma', L, V)$ are such that $L_{12}=L_{34}$  has codimension at least $2$, 
where $L_{ij}$ is the line determined by the images of the points $i,j$ via the morphism $\Sigma'\ra\PP^2$ given by the 
complete linear system $|L|$. Then the exceptional divisor of $\pi_{\Gamma\cup\{n+1\}}$ has a unique irreducible component  intersecting $M_{0,n+1}$. 
\end{Proposition}

We can describe the extremal divisor $D$ in $\oM_{0,n}$ using projections
from the canonical model of a stable hypergraph curve:

\begin{Theorem}\label{canonical}
In the setup of Theorem \ref{NewDivisors}, let $\Sigma^s$ be the stable hypergraph curve.
The canonical invertible sheaf $\omega_{\Sigma^s}$ is very ample.
The canonical embedding $\Sigma^s\subset\bP^{n-4}$ is a union of lines.
Project $\Sigma^s$ away from points on $n-6$ lines of $\Sigma^s\setminus\Sigma$
(one point on each component).
Assume that, for a sufficiently general choice of projection points,
this gives a morphism $g:\Sigma\dra\bP^2$ and that $g^*\O(1)$ is admissible.

Then we have an induced rational map
\begin{equation}\label{Psi}
\psi:(\PP^1)^{n-6}\dra W^2(\Sigma).
\cooltag\end{equation}
with a non-empty domain.
The closure $W$ of the image of $\psi$ is an irreducible component of $W^2(\Sigma)$
and the closure in $\M_{0,n}$ of the preimage of $W$ in $G^2(\Sigma)$ is an extremal divisor of $\M_{0,n}$. 
Under the stronger assumptions of Proposition \ref{Irred}, $W=W^2(\Sigma)$ and we recover the extremal divisor  
$D$ of Theorem \ref{NewDivisors}.  
\end{Theorem} 

We give sufficient conditions for the assumptions in Theorem \ref{canonical}:
\begin{Proposition}\label{ExtraCond}
In the setup of Theorem \ref{NewDivisors}, let $\Sigma^{\alpha}$ be the hypergraph curve corresponding to the hypergraph obtained from $\Gamma$ by 
removing the hyperedge  $\Gamma_{\alpha}$. Assume $W^3(\Sigma^{\alpha})=\eset$ for all $\alpha=1,\ldots,n-2$. Then 
the map $\psi$ of (\ref{Psi}) has a non-empty domain.
\end{Proposition}

Proofs of Theorems  \ref{NewDivisors} and  \ref{canonical}, Propositions \ref{Irred} and \ref{ExtraCond} will be given 
in~\ref{divisorssection}.

\begin{Example}
Theorem~\ref{canonical} gives the following model of the Keel--Vermeire divisor:
it is the closure in $\oM_{0,6}$ of the locus of points obtained by projecting
singular points of $\Sigma$ from Fig.~\ref{KHyper} from points in $\bP^2$
(note that $\Sigma$ has genus $3$ and Fig.~\ref{KHyper} is its canonical embedding
as a degree $4$ curve). 

Consider Fig.~\ref{Tetrahedron}. Here $\Sigma^s$ has genus~$8$
and canonically embeds in $\bP^7$. Projecting its away from points of 
$5$ components of $\Sigma^s\setminus\Sigma$ gives a $5$-parameter family of morphisms $\Sigma\to\bP^2$.
Projecting singular points of $\Sigma$ from points of $\bP^2$ gives a $7$-dimensional locus in $\oM_{0,11}$,
an extremal divisor. This example is interesting because $W^2=W^3$.
So the hypergraph morphism $\pi_\Gamma:\,\oM_{0,12}\to(\bP^1)^9$ has a unique exceptional divisor~$D$
intersecting $M_{0,12}$ and fibers of the restriction $\pi_\Gamma|_{D}$
are $2$-dimensional.
\end{Example}

We do not know how to classify all $3$-graphs that satisfy conditions of Theorem~\ref{NewDivisors}.
Moreover, one can show that a random $3$-graph
does not satisfy these conditions with probability that tends to $1$ as $n\to\infty$.
Indeed, one of the main results in the theory of random hypergraphs
is that they are almost surely disconnected. 
It is easy to see that disconnected hypergraphs do not satisfy 
neither $(\dag)$ nor even \eqref{CondS}. 
The hypergraph of Fig.~\ref{Tetrahedron} has $W^3\ne\emptyset$.
So in a sense our hypergraphs are ``rare''.
To convince the reader that they exist, 
let us give a Fibonacci-like
construction\footnote{We are grateful
to Anna Kazanova who suggested this construction to us.} of 
$3$-graphs that satisfy the conditions of Theorem~\ref{NewDivisors}, Proposition \ref{Irred}, and even 
Proposition \ref{ExtraCond}.
In fact, it shows that their number 
grows very rapidly as $n$ goes to infinity.

\begin{Review}\textsc{Construction.}\label{Fibonacci}
We start when $n=6$ with a hypergraph of Fig.~$1$.
Note that the last two triples contain $n$ and the last triple contains $n-1$.
We will keep this property in the inductive construction.
The inductive step:
suppose we have a collection of $n-2$ triples $\Gamma'_1,\ldots,\Gamma'_{n-2}$ for $k=n$.
We define $n-1$ triples for $k=n+1$ as follows: $\Gamma_i:=\Gamma'_i$ for $i=1,\ldots,n-3$;
if $\Gamma'_{n-2}=\{i,n-1,n\}$ then we define $\Gamma_{n-2}:=\{i,n-1,n+1\}$;
and we define $\Gamma_{n-1}:=\{a,n,n+1\}$, where $a$ is any index in $N\setminus(\Gamma_{n-2}\cup\Gamma_{n-3})$, see Fig.~\ref{constructionne}.
\begin{figure}[htbp]
\includegraphics[width=4in]{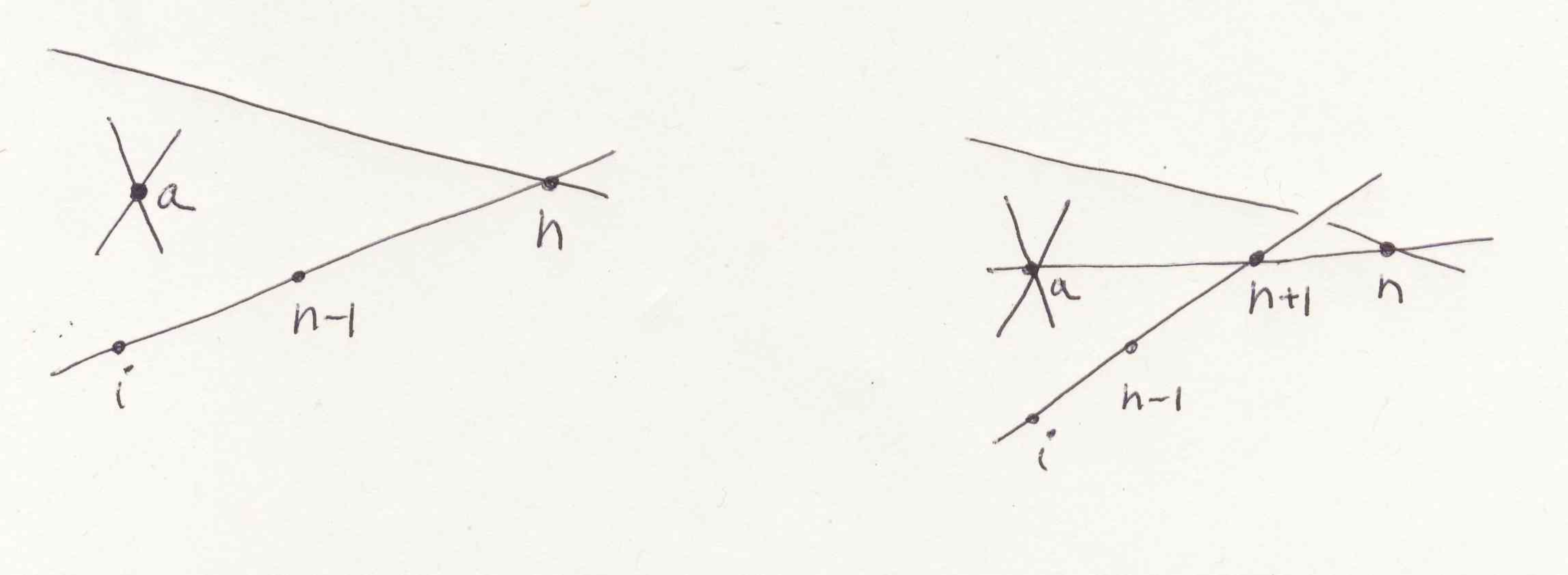}
\caption{}\label{constructionne}
\end{figure}
\end{Review}

\begin{Theorem}\label{FunnyConstruction} 
$\Gamma$ constructed in \ref{Fibonacci} satisfies all conditions of Theorem~\ref{NewDivisors}, Proposition \ref{Irred} and Proposition \ref{ExtraCond}.
\end{Theorem}

\begin{proof}
Firstly, we show that  $\Gamma$ satisfies $(\dag)$.
Suppose $I\subset\{1,\ldots,n-1\}$, $1<|I|<n-1$.
Consider several cases. If $I\subset\{1,\ldots,n-3\}$ then 
$$|\bigcup_{i\in I}\Gamma_i|=|\bigcup_{i\in I}\Gamma'_i|\ge |I|+3$$
by inductive assumption, and we are done.
If $I=I'\cup\{n-2\}$ (resp.~$I=I'\cup\{n-1\}$), where $I'\subset\{1,\ldots,n-3\}$, then
$$|\bigcup_{i\in I}\Gamma_i|\ge|\bigcup_{i\in I'}\Gamma'_i|+1,$$
because $n+1$ belongs to the first union but does not belong to the second union.
So again we are done by the inductive assumption, unless $|I'|=1$, in which case the claim is easy.
So it remains to consider the case 
$I=I'\cup\{n-2, n-1\}$, where $I'\subset\{1,\ldots,n-3\}$ (and note that $|I'|<n-3$).
If $I'$ is empty then the claim is easy. Otherwise, let
$I''=I'\cup\{n-2\}$. Then $1<|I''|<n-2$ and therefore, by inductive assumption,
$$|\bigcup_{i\in I''}\Gamma'_i|\ge|I''|+3.$$
But $$\bigcup_{i\in I}\Gamma_i\supseteqq\bigcup_{i\in I''}\Gamma'_i\sqcup\{n+1\}.$$
This proves that  $\Gamma$ satisfies $(\dag)$.

It is clear from construction that $W^2(\Sigma)\ne\emptyset$.
Let us  show that $W^3(\Sigma')$ is empty.
Let $f:\,\Sigma'\to\bP^3$ be a morphism linear on components of $\Sigma'$
and separating singular points. Let $P_1,\ldots,P_{n}$ be their images 
(for $k=n+1$), i.e., $P_a,P_b,P_c$ are collinear for each $\{a,b,c\}=\Gamma_i$, 
$i=1,\ldots,n-3$.
We claim (arguing by induction) that these points lie on a plane.
Note that, by construction, $\Gamma_{n-3}=\{u,v,n\}$ is the only triple on this list that contains~$n$.
Therefore $P_1,\ldots,P_{n-1}$ lie on a plane by inductive assumption, and so 
$P_n$ lies on the same plane.

We prove now that the locus $W$ in $W^2(\Sigma')$ where $L_{12}=L_{34}$ has codimension at least $2$. 
Note, by Theorem \ref{ID}, $G^2(\Sigma')=G^1(\Sigma')\cong M_{0,n-1}$; in particular, $G^2(\Sigma')$ is irreducible. It follows that 
$W^2(\Sigma')$ is irreducible (the map $G^2(\Sigma')\ra W^2(\Sigma')$ is surjective). Moreover, $\dim W^2(\Sigma')=n-4$.

We prove by induction on $n$ that: (1) If $a,b$, and $c,d$ are two pairs of distinct indices in $\{1,\ldots,n-1\}$
not contained in $\Gamma_j$, for any $j=1,\ldots,n-4$, then the locus $W$ in $W^2(\Sigma')$ where $L_{ab}=L_{cd}$ 
has codimension at least $2$; (2) If  $a,b,c\in\{1,\ldots,n-1\}$ (distinct) with the pair $a,b$ not contained in $\Gamma_j$, for any $j=1,\ldots,n-4$,  
then the locus $V$ in $W^2(\Sigma')$ where $c\in L_{ab}$ has codimension at least $1$.

Let $\Sigma''$ be the hypergraph curve corresponding to $$\Ga''=\{\Gamma_1,\ldots,\Gamma_{n-5}\}$$ (for the case $n-1$ this is the curve
$\Sigma'$). There is a well-defined morphism $\phi:W^2(\Sigma')\ra~W^2(\Sigma'')$ given by restriction. 
From the inductive construction, it is clear that $\Gamma_{n-4}$ is the only triple among $\Gamma_1,\ldots,\Gamma_{n-4}$ that contains $n-1$.
Let $\Gamma_{n-4}=\{n-1,e,f\}$. Note, $\{e,f\}\neq\{a,b\}$. 
The fibers of $\phi$ are irreducible, $1$-dimensional (given by the point $n-1$ moving on the line $L_{ef}$).

Proof of (1): If $a,b,c,d\neq n-1$,  we are done by induction, as $W$ is the preimage under $\phi$ of the similarly defined locus in $W^2(\Sigma'')$.
Assume without loss of generality that $d=n-1$. There are several cases for the components $W_0$ of $W$. If $L_{ab}=L_{ef}$ for all geometric points
of $W_0$, then $W_0$ is contained in the pull-back of codimension $\geq2$ locus in $W^2(\Sigma'')$ (by induction). Therefore, we may assume that 
$L_{ab}\neq L_{ef}$ for a general geometric point of $W_0$. There are two cases. Case (i): If $a,b,e,f$ are all distinct (for a general point of $W_0$). 
Then $d\in L_{ab}$ imposes one condition on elements of $W^2(\Sigma')$ (a non-empty general fiber of $\phi|_{W_0}$ is a point). 
In this case, $\phi(W_0)$ is contained in the locus in  $W^2(\Sigma'')$ where $c\in L_{ab}$. By the induction assumption for (2),  
it follows that $\phi(W_0)$ has codimension at least $1$ in $W^2(\Sigma'')$; hence, $W_0$ has codimension at least $2$ in $W^2(\Sigma')$. 
Case (ii): If $a,b,e,f$ are not all distinct. We may assume $a=e$, and $a,b,f$ distinct. Then $d\in L_{ab}$ implies
$d=L_{ab}\cap L_{ef}=a$, contradiction (not in $W^2(\Sigma')$). 

Proof of (2): As in the proof of (1), we may assume $c=n-1$. Let $V_0$ be a component of $V$. If $a,b,e,f$ are not all distinct (along $V_0$), then our assumptions imply that 
one of the triples $e,a,b$ or $f,a,b$ consists of distinct points and we are done by induction. Assume $a,b,e,f$ are all distinct (for a general point of $V_0$). If
$L_{ab}=L_{ef}$ for all points of $V_0$, then $V_0$ is contained in the locus in $W^2(\Sigma')$ where $e,f\in L_{ab}$ and we are again done by induction. If 
$L_{ab}\neq L_{ef}$ for a general point of $V_0$, then $c\in L_{ab}$ imposes one condition on elements of $W^2(\Sigma')$ as the non-empty general fiber of 
$\phi|_{V_0}$ is a point. 

We prove now that $\Gamma$ satisfies the condition of Proposition \ref{ExtraCond}. This is clear if $n=6$. We assume $n\geq7$. If $\alpha=n-2$ or $n-3$, then $\Sigma'\hra\Sigma^{\alpha}$ induces by restriction a morphism $W^3(\Sigma^{\alpha})\ra W^3(\Sigma')$.  It follows that $W^3(\Sigma^{\alpha})=\eset$. Hence, we may assume $\alpha\neq n-2,n-3$. Let $$\Gamma_{n-2}=\{i,n-1,n\},\quad\Gamma_{n-3}=\{j,n-2,n\}.$$ 

Let $\Sigma'^{\alpha}$ be the hypergraph curve corresponding to the hypergraph obtained from $\Gamma$ by 
removing the hyperedges $\Gamma_{\alpha},\Gamma_{n-3},\Gamma_{n-2}$. Note that $W^3(\Sigma'^{\alpha})=\eset$ implies that $W^3(\Sigma^{\alpha})=\eset$, if the inclusion $\Sigma'^{\alpha}\hra\Sigma^{\alpha}$ induces a  morphism $W^3(\Sigma^{\alpha})\ra W^3(\Sigma')$. This is clearly the case if 
at least three of $i,j,n-2,n-1$ are contained in $$S=\bigcup_{u\neq\alpha,n-3,n-2}\Gamma_{u}.$$ 

We claim that this is always the case. Assume $u,v\in\{i,j,n-2,n-1\}$ are not contained in  $S$. Then $S\subset\{1,\ldots,n-1\}\setminus\{u,v\}$ and therefore $|S|\leq n-3$, which contradicts $(\dag)$ (use $n\geq7$).

We prove by induction on $n\geq7$ that $W^3(\Sigma'^{\alpha})=\eset$ and $W^3(\Sigma^{\alpha})=\eset$. This is clearly true for $n=7$. 
Assume $n\geq8$. If $\alpha=n-4$ then we are done, as $\Sigma'^{\alpha}=\Sigma''$ (the curve $\Sigma'$ for the case $n-1$). 

Assume $\alpha\leq n-5$. Denote by $\Sigma^{\alpha}_{n-1}$ the hypergraph curve corresponding to the hypergraph obtained from $\{\Gamma_1,\ldots,\Gamma_{n-4},\{v,n-2,n-1\}\}$ (defining a curve in our construction, for $n-1$), by removing the hyperedge $\Gamma_{\alpha}$. By the induction assumption, 
$W^3(\Sigma^{\alpha}_{n-1})=\eset$. We claim that the inclusion $\Sigma'^{\alpha}\hra\Sigma^{\alpha}_{n-1}$ induces a morphism 
$W^3(\Sigma^{\alpha}_{n-1})\ra W^3(\Sigma'^{\alpha})$. As $n-1\in\Gamma_{n-4}$ and hence, $n-1\in S$ ($\alpha\leq n-5$), this is clear if $v$ or $n-2$ belong to $S$. 
But this is indeed the case, as $v,n-2\in\bigcup_{u=1}^{n-4}\Gamma_{u}$ and they cannot be both contained in $\Gamma_{\alpha}$ by $(\dag)$.

\end{proof}

The number of possibilities for adding an extra vertex $a$ grows rapidly with $n$
but for $n=6,7,8$ there is just one possibility (up to symmetries), see Fig.~\ref{678}. 
However, for $n=8$, there exist extremal divisors, depicted on  Fig.~\ref{another8},
that cannot be obtained by an inductive construction. The first one shows the only possibility 
(up to symmetries) of having six triples that satisfy $(\dag)$ (other than the one coming from the inductive 
construction) and the conditions in Theorem \ref{NewDivisors} and Propositions \ref{Irred} and \ref{ExtraCond}. 
The second one shows a collection $\Gamma_1,\ldots,\Gamma_5$ (one of which is a $4$-tuple) that satisfy $(\dag)$ and Corollary \ref{manyrays}. One can also show that this divisor is not a pull-back from $\oM_{0,7}$.

\begin{figure}[htbp]
\includegraphics[width=5in]{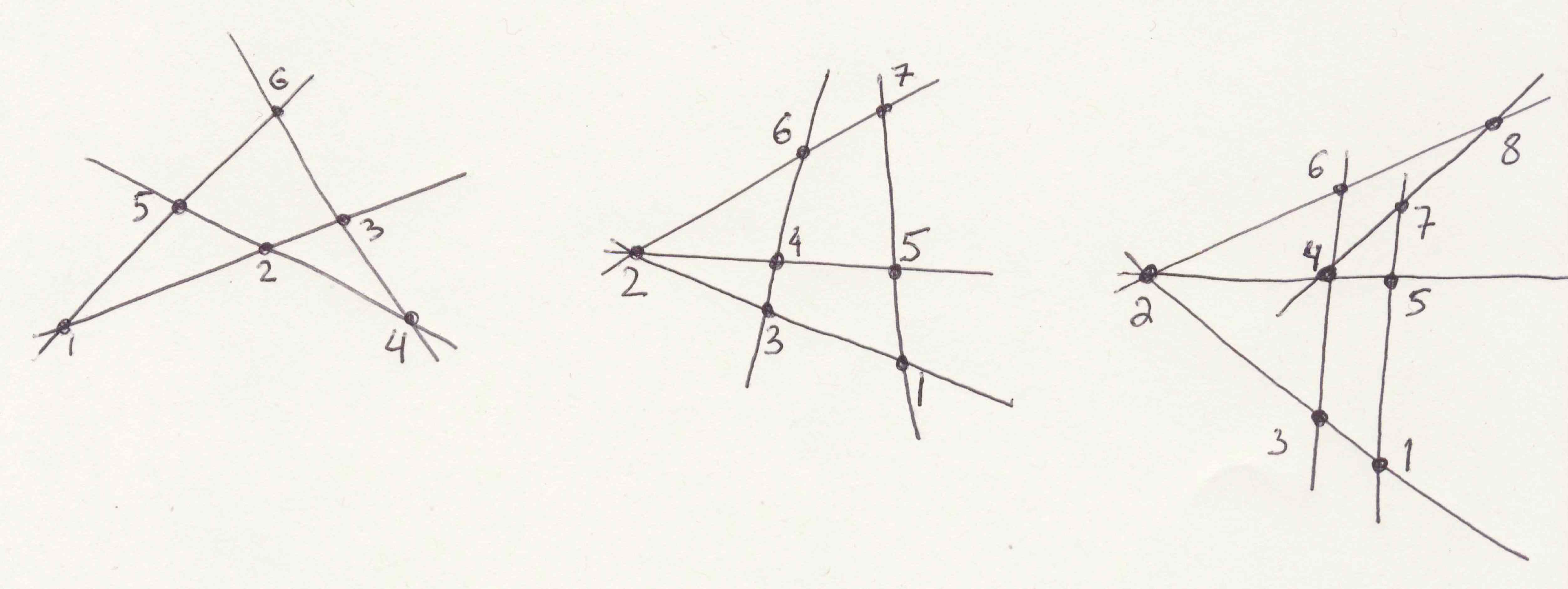}
\caption{\small Extremal divisors on $\oM_{0,6}$, $\oM_{0,7}$, $\oM_{0,8}$ obtained by Construction~\ref{Fibonacci}.
}\label{678}
\end{figure}

\begin{figure}[htbp]
\includegraphics[width=5in]{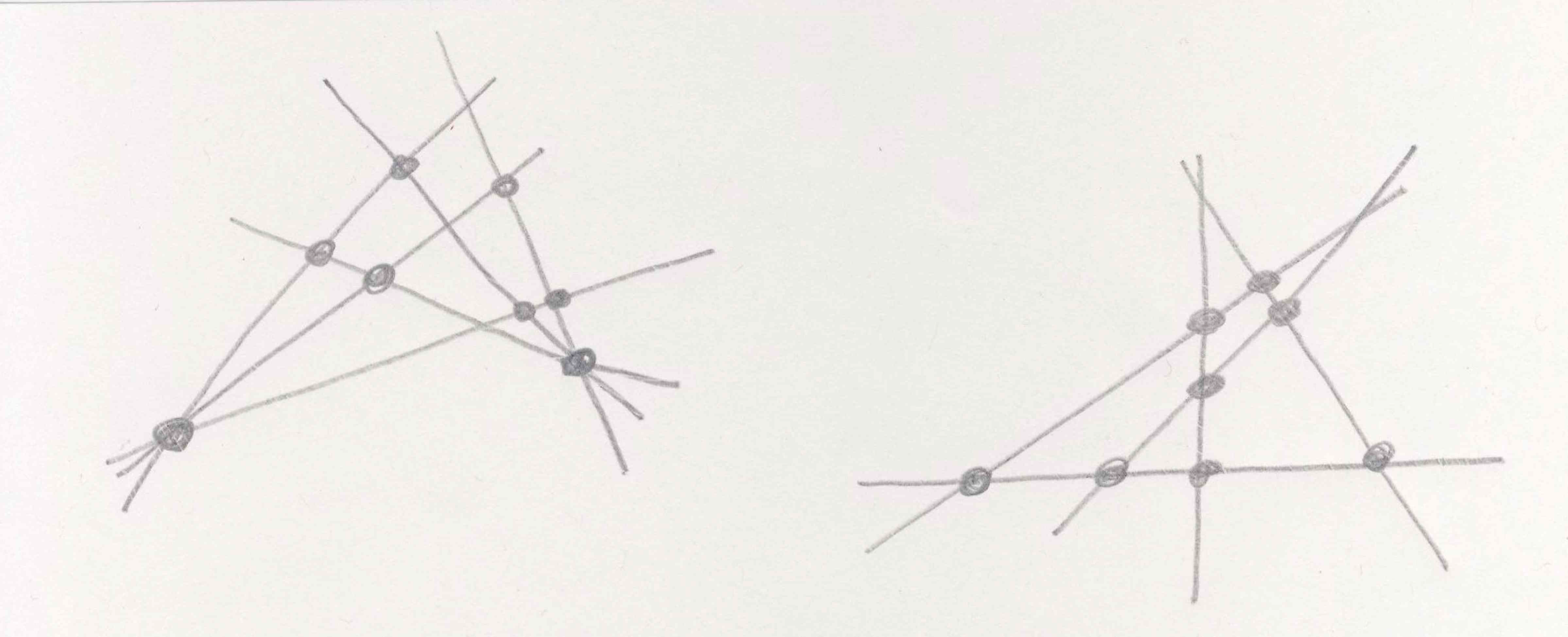}
\caption{\small Two more extremal divisors on~$\oM_{0,8}$.
}\label{another8}
\end{figure}


\section{Admissible Sheaves on Hypergraph Curves}\label{smaps}

What is the geometric meaning of ``strange'' conditions \eqref{CondS}
$$
|\bigcup_{j\in S}\Gamma_j|\ge|S|+2
\quad\hbox{\rm for any $S\subset\{1,\ldots,n-2\}$}
$$
and \eqref{CondM}
$$|\bigcup_{i\in S}\Gamma_i|\ge|S|+3
\quad\hbox{\rm for any $S\subset\{1,\ldots,n-2\}$ with $2\le|S|\le n-3$}$$
(assuming for simplicity that all hyperedges are triples) of Section~\ref{DivSection}?
We will interpret these conditions as statements
about admissible sheaves on~$\Sigma^s$.
But first let us take a look at the morphism 
\begin{equation}\label{morphismV}
v:\,M_{0,n}\to \Pic^{\uone}.
\cooltag\end{equation}
It plays only an auxiliary role in the rest of the paper
because (as it is easy to show) it does not
extend to the morphism from $\oM_{0,n}$ to any compactification of $\Pic^{\uone}$,
and of course $\oM_{0,n}$ is the main subject of this paper.
In fact, even $M_{0,n}$ is clearly a ``wrong'' domain for the morphism $v$.
The condition (a) in the definition of the functor $\cG^1$ seems unnatural
and we can get rid of it:

\begin{lemdef}
The following functors $Schemes\to Sets$ are equivalent:
\begin{itemize}
\item
A scheme $S$ goes to  the set of isomorphism classes
of data $(\bSi, f:\,\bSi\to\bP^1_S)$, where 
$\bSi\in\cM_\Gamma(S)$ and the restriction of $f$
to  each irreducible component of $\bSi$ is an isomorphism.
\item
A scheme $S$ goes to the set of isomorphism classes of flat families $C\to S$
with $n$ sections $s_1,\ldots,s_n$ such that (a) any geometric fiber is isomorphic to $\bP^1$;
(b) for any $\Gamma_\alpha$, sections $s_i$ with $i\in\Gamma_\alpha$ are disjoint.
\end{itemize}
This functor $\cM or^{\uone}$ is represented by a smooth quasiprojective scheme
$\Mor^{\uone}$.
We have natural morphisms 
$$M_{0,n}\hookrightarrow\Mor^{\uone}\arrow^v\Pic^{\uone},$$
where $v$ is the pull-back of $\cO_{\bP^1}(1)$.
\end{lemdef}

\begin{proof}
The equivalence of functors is established as in the proof of Theorem~\ref{ID}.
Note that the second functor is isomorphic to the quotient of an open locus (with a divisorial boundary) in $(\bP^1)^n$
by the free action of $\PGL_2$. 
One can eliminate the action altogether by fixing (for example) the first three sections in $\Gamma_1$
to be $0$, $1$, and~$\infty$.
This shows that $\cM or^{\uone}$
is represented by a smooth quasiprojective scheme.
\end{proof}

The relation to $\oM_{0,n}$ is as follows.

\begin{Proposition}\label{relationes}
Let 
$\oM_{0,n}^{adm}\subset\oM_{0,n}$
be the complement
to the union of boundary divisors $\delta_I$ such that 
for some $\alpha$, $\beta$, 
$$I\cap\Gamma_\alpha>1\quad\hbox{\rm and}\quad I^c\cap\Gamma_\beta>1.$$
There is a natural morphism 
$$\oM_{0,n}^{adm}\to\Mor^{\uone}$$
extending the embedding of $M_{0,n}$.
\end{Proposition}

\begin{proof}
Recall that $\oM_{0,n+1}\to\oM_{0,n}$ is the universal family of $\oM_{0,n}$.
 For a family $C\to S$ of stable $n$-pointed rational curves with sections $s_1,\ldots,s_n$,
consider the induced morphism 
$$F_\alpha:\,C\to\oM_{0,n+1}\arrow^{\Gamma_\alpha\cup\{n+1\}}\oM_{\Gamma_\alpha\cup\{n+1\}}\simeq\bP^1.$$
Consider the families $\bP^1_S\to S$ with sections $F_\alpha(s_1),\ldots,F_\alpha(s_n)$.

We claim that if the induced morphism $S\to\oM_{0,n}$ factors through $\oM_{0,n}^{adm}$
then any of these families (for $\alpha=1,\ldots,d$) gives an object of $\cM or^{\uone}$
(these families are isomorphic for different~$\alpha$'s and therefore give the same object).
This gives a natural transformation $\ocM_{0,n}^{adm}\to \cM or^{\uone}$ we are looking for.

It suffices to check this on geometric points $S=\Spec\Omega$. The $C/\Omega$ is a stable rational curve
and the morphism $F_\alpha:\,C\to\bP^1_\Omega$ can be described as follows. in the dual graph of~$C$,
consider the minimal subtree with legs marked by $\Gamma_\alpha$.
Since $I\cap\Gamma_\alpha>1$ and $I^c\cap\Gamma_\alpha>1$ for any boundary divisor $\delta_I$,
this subtree in fact has just one vertex. The same argument shows that in fact this subtree is independent of $\alpha$.
The morphism $F_\alpha$ just collapses $C$ (and its marked points)
to the irreducible component of $C$ that corresponds to the vertex of the minimal subtree.
\end{proof}

Let us work out $v$ in coordinates. We don't need this formula but it is too simple and nice to ignore.
We have 
$$\Pic^{\underline0}=\H^1(\Pi, \bG_{m,M_\Gamma}),$$
where $\Pi$ is the dual graph of $\Sigma^s$.
So morphisms of $k$-schemes $\Mor^{\uone}\to \Pic^{\uone}$ are classified,
up to the action of $\Pic^{\underline0}$ on $\Pic^{\uone}$, by the group
\begin{equation}\label{fungroup}
\Hom(\H_1(\Pi,\bZ),\ \cO^*(\Mor^{\uone})/\pi_\Gamma^*\cO^*(M_\Gamma)).
\cooltag\end{equation}

\begin{Lemma}
Let $\gamma'=\{\gamma'_1\to\ldots\to\gamma'_{s'}\to\gamma'_1\}$ be a cycle
in $\Pi$, i.e.,~a closed chain of $\bP^1$'s  in $\Sigma^s$ (possibly with repetitions).
Let 
\begin{equation}\label{chaingamma}
\gamma=\{\gamma_1\to\ldots\to\gamma_{s}\to\gamma_1\}
\cooltag\end{equation} 
be its image in $\Sigma$.
Let $a_i=\gamma_i\cap\gamma_{i+1}$ (modulo $s$) and let $b_i\in\gamma_i$ be any
singular point of $\Sigma$ different from $a_i$ and $a_{i-1}$ (modulo $s$), see Fig.~\ref{cycles}.
\begin{figure}[htbp]
\includegraphics[width=2in]{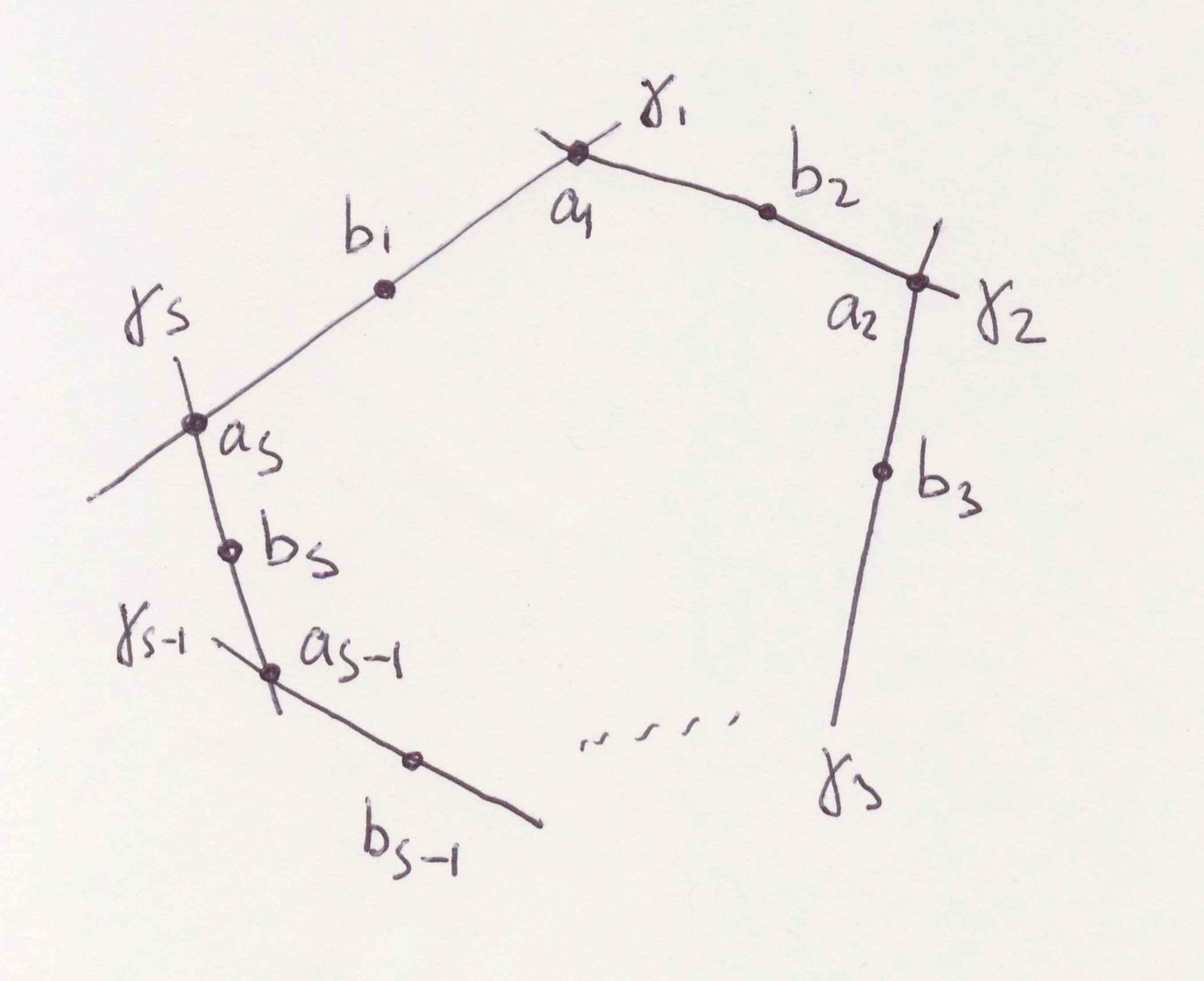}
\caption{}\label{cycles}
\end{figure}
The morphism $v$ of~\eqref{morphismV} corresponds to a linear functional in \eqref{fungroup}
that sends $[\gamma']\in H_1(\Pi,\bZ)$  to 
\begin{equation}\label{hubbabubba}
{x_{b_1}-x_{a_1}\over x_{b_1}-x_{a_s}}{x_{b_2}-x_{a_2}\over x_{b_2}-x_{a_1}}
\dots{x_{b_s}-x_{a_s}\over x_{b_s}-x_{a_{s-1}}}\in\cO^*(\Mor^{\uone})/\pi_\Gamma^*\cO^*(M_\Gamma), 
\cooltag\end{equation}
where we represent an $S$-point of $\Mor^{\uone}$ by an $n$-tuple $x_1,\ldots,x_n\in\bP^1(S)$.
\end{Lemma} 

\begin{proof}
Let $C$ be a stable genus $1$ curve with $s$ rational components given by  \eqref{chaingamma}
and marked by $b_1,\ldots,b_s$ (it is unique up to an isomorphism). 
Note that for any flat family of hypergraph curves $\bSi\to S$,
we have a natural morphism $j:\,C\times S\to\bSi$.
By functorialities, it suffices to prove that the morphism 
$$\Mor^{\uone}\arrow^v\Pic^{\uone}(\Sigma)\arrow^{j^*}\Pic^{\uone}(C)\arrow^{\otimes\cO_C(-b_1-\ldots-b_s)}\Pic^{\uzero}(C)=\bG_m$$
is given by \eqref{hubbabubba}.
For any $i=1,\ldots,s$, fix an isomorphism of $\gamma_i$ with $\bP^1$
by sending $a_{i-1}\mapsto 0$, $a_i\mapsto\infty$, and $b_i\mapsto 1$.
If we ignore markings $b_1,\ldots,b_s$ then $\Aut(C)\simeq	\bG_m^s$, where the $i$-th copy of $\bG_m$
acts only on $\gamma_i\simeq\bP^1$ in the standard way (preserving $0$ and $\infty$).
It is immediate from definitions that $\Aut(C)$ acts on $\Pic^{\uzero}(C)$ through the product
homomorphism $\bG_m^s\to\bG_m$.
For any object $(\bSi\to S,\ f:\,\bSi\to\bP^1_S)$ of $\cG^1(S)$, consider the composition
$C\times S\arrow^{f\circ h}\bP^1_S$. Then 
$(f\circ h)^{-1}(\infty)$ has ${x_{b_i}-x_{a_i}\over x_{b_i}-x_{a_{i-1}}}$
as its $i$-th component.
Therefore, 
$$
\left ({x_{b_1}-x_{a_1}\over x_{b_1}-x_{a_s}},\ {x_{b_2}-x_{a_2}\over x_{b_2}-x_{a_1}},\ \ldots,\ 
{x_{b_s}-x_{a_s}\over x_{b_s}-x_{a_{s-1}}}\right)\in \Aut(C)(S)
$$
translates the divisor $(b_1,\ldots,b_s)$ to~$(f\circ h)^{-1}(\infty)$
and the Lemma follows.
\end{proof}

\begin{Example} 
Our running example will be the Keel--Vermeire curve of Fig.~\ref{KHyper}.
The morphism $G^1(\Sigma)=M_{0,6}\arrow^v\bG_m^3$ collapses the 
Keel--Vermeire divisor
$$G^2(\Sigma)\to W^2(\Sigma)=\{\omega_\Sigma\}.$$ 
We read a basis of 1-cycles from Fig.~\ref{KHyper}:
$1\to 5\to 2\to 1$, $4\to 2\to 3\to 4$, and $5\to 4\to6\to5$.
This shows that 
$v((x_1,\ldots,x_6)\mod PGL_2)$ is equal to
\begin{equation}\label{vforKV}
\left({x_6-x_5\over x_6-x_1}{x_4-x_2\over x_4-x_5}{x_3-x_1\over x_3-x_2},\quad
{x_5-x_2\over x_5-x_4}{x_1-x_3\over x_1-x_2}{x_6-x_4\over x_6-x_3},\quad
{x_2-x_4\over x_2-x_5}{x_3-x_6\over x_3-x_4}{x_1-x_5\over x_1-x_6}
\right)
\cooltag\end{equation}

In this case, the morphism of Prop.~\ref{relationes} is an isomorphism:
$\Mor^{\uone}$ is isomorphic to the complement in $\oM_{0,6}$
to the union of boundary divisors other that $\delta_{14}$, $\delta_{26}$, and $\delta_{35}$.
Or, more concretely, by fixing $x_2\mapsto0$, $x_4\mapsto 1$, $x_5\mapsto\infty$, $\Mor^{\uone}$
is isomorphic to the complement in  
$$(\bP^1_{x_1}\setminus\{0,\infty\})\times(\bP^1_{x_3}\setminus\{0,1\})\times(\bP^1_{x_6}\setminus\{1,\infty\})$$
to the union of diagonals.
The morphism \eqref{vforKV}
takes the form
\begin{equation}\label{vforKV}
\left({1\over x_6-x_1}{x_3-x_1\over x_3},\quad
{x_1-x_3\over x_1}{x_6-1\over x_6-x_3},\quad
{x_3-x_6\over x_3-1}{-1\over x_1-x_6}
\right)
\cooltag\end{equation}
\end{Example}

\begin{Review}\textsc{Setup.}
For simplicity, from now on 
we work exclusively in the ``divisorial'' setup of Section~\ref{DivSection}:
assume that $\Gamma$ has $d=n-2$ triples
and that all valences are equal to $2$ or~$3$. 
Neither $\Sigma$ nor $\Sigma^s$ have moduli 
and the dualizing sheaf $\omega_{\Sigma^s}$ has degree $1$ on each component of $\Sigma^s$.We have 
$$\dim M_{0,n}=n-3=g=\dim \Pic^{\uone}(\Sigma).$$

Moreover, we are going to impose the condition \eqref{CondS}
(recall that it is equivalent to $W^1(\Sigma)\ne W^2(\Sigma)$, 
i.e.~to $v$ being birational).
\end{Review}

The torsor $\Pic^{\uone}(\Sigma)$ has a natural proper (but reducible) model $\oPic^{n-2}$,  the 
 ``compactified Jacobian'' of~\cite{OS}, see \cite{Ca} and \cite{Al} for more recent developments.
Its geometric points 
correspond to gr-equivalence classes of admissible sheaves:

\begin{Definition}
A coherent sheaf on $\Sigma^s$ is called {\em admissible}
if it is torsion-free, has rank~$1$ at generic points of $\Sigma^s$, has degree $n-2$, 
and is semi-stable with respect to the canonical polarization~$\omega_{\Sigma^s}$.
\end{Definition}

\begin{Review}
$\oPic^{n-2}$ is a stable toric variety of $\Pic^{\uzero}(\Sigma)$, in partucular its normalization
is a disjoint union of toric varieties. 
By \cite{Si}, it is functorial in the following sense: consider the functor 
$\ocPic^{n-2}:\, Schemes\to Sets$
that assigns to a scheme $S$ the set of coherent sheaves on $\Sigma'\times S$ flat over $S$
and such that its restriction to any geometric fiber $\Sigma'_\Omega$ is admissible.
Then there exists a natural transformation 
$\ocPic^{n-2}\to h_{\oPic^{n-2}}$ 
which has the universal property: for any scheme $T$, any natural 
transformation 
$\ocPic^{n-2}\to h_T$ factors through a unique morphism $\oPic^{n-2}\to T$.
\end{Review}

\begin{Proposition}\label{stabledegree}
Any invertible sheaf $L\in \Pic^{\uone}$ is stable, i.e.
$$\Pic^{\uone}\subset\oPic^{n-2}.$$
\end{Proposition}

\begin{proof}
We can use the well-known Gieseker's ``basic inequality'': 
an invertible sheaf $F$ on a stable curve $X$ is stable w.r.t. the canonical polarization $\omega_X$
if and only if 
\begin{equation}\label{GBI}
|\deg (F|_Y)-\lambda_Y\deg F|<{1\over2}|Y\cap\ov{ X\setminus Y}|
\cooltag\end{equation}
for every proper non-empty subcurve $Y\subset X$, where 
$$\lambda_Y={\sum\limits_{X_i\subset Y}\deg \omega_X|_{X_i}\over\deg\omega_X},$$ 
the sum over irreducible components of $Y$.

The dual graph $\Pi$ of $\Sigma'$ has $n-2$ {\em black} vertices that correspond to components of~$\Sigma$
and $n-6$ {\em white} vertices that correspond to $\bP^1$'s inserted at triple points of~$\Sigma$. 
It is ``almost'' bipartite for large $n$:
each white vertex is connected only to black vertices and there are exactly $6$ edges
connecting pairs of black vertices.

By swapping $Y$ and $\Sigma^s\setminus Y$ if necessary, we can assume that  
\begin{equation}\label{diamond}
(n-2)w(Y)-(n-6)b(Y)\ge0,
\cooltag\end{equation}
where $b(Y)$ (resp.~$w(Y)$) is the number of black (resp.~white) components in $Y$.
A simple calculation shows that  \eqref{GBI} is then equivalent to
\begin{equation}\label{lucy}
(n-2)w(Y)-(n-6)b(Y)-(n-4)\#(Y)<0,
\cooltag\end{equation}
where $\#(Y)=|Y\cap\ov{\Sigma^s\setminus Y}|$.
Suppose that $Y$ contains a black component $B_1$ adjacent to a white component $W_1$ of $\Sigma^s\setminus Y$.
Consider a new subcurve $Y_1$ obtained by adding $W_1$ to $Y$.
Note that the LHS in \eqref{diamond} (resp. \eqref{lucy} )
only increases (resp. decreases) when we pass from $Y$ to $Y_1$.
So it suffices to prove \eqref{lucy} for $Y_1$. Doing this as many times as necessary, we can 
assume without loss of generality that all white components of $\Sigma^s$ adjacent to black components of $Y$
also belong to $Y$.

Let $e(Y)$ be the number of white components of $Y$ adjacent to $3$ black components of $\Sigma^s\setminus Y$.
Let $c(Y)$ (resp.~$r(Y)$) be the number of singular points (resp.~nodes) of $\Sigma$ covered by the image of $Y$.
Then we have 
$$w(Y)=e(Y)+c(Y)-r(Y)$$
and (by counting nodes in $Y$)
$$\begin{array}{rl} 
3b(Y)\ =&3(w(Y)-e(Y))+2r(Y)-(\#(Y)-3e(Y))\cr
=&3w(Y)+2r(Y)-\#(Y).\end{array}$$
Using these equations to find $\#(Y)$ and $w(Y)$, we can 
rewrite the LHS in \eqref{lucy} as
$$
-(2n-10)e(Y)-(2n-10)c(Y)-2r(Y)+(2n-6)b(Y).
$$
Therefore, it suffices to prove that
$$
(n-5)c(Y)+r(Y)-(n-3)b(Y)\ge0.
$$
By  \eqref{CondS}, $c(Y)\ge b(Y)+2$. Therefore, it suffices to prove that
\begin{equation}\label{finaleq}
(n-5)+{r(Y)\over 2}\geq b(Y).
\cooltag\end{equation}
If $b\le n-5$ the \eqref{finaleq} is clear.
If  $b=n-4$, then $r(Y)\ge 5$, and if $b=n-3$ then $r(Y)=6$.
In all these cases \eqref{finaleq} also holds.
\end{proof}

\begin{Review}\label{CaporasoStuff}
Let $\oPicone$ be the closure of $\Pic^{\uone}$ in $\oPic^{n-2}$.
It is a (possibly non-normal) toric variety of $\Pic^{\uone}$. Its toric strata of codimension $k$
can be described as follows \cite{Ca}.
Choose $k$ nodes in $\Sigma'$ and let $\hat\Sigma$ be a curve obtained from $\Sigma'$
by inserting $k$ strictly semistable $\bP^1$'s at the chosen nodes.
Start with the multidegree $\uone$ and 
choose an arbitrary multidegree $\hat d$ on $\hat\Sigma$ such that the degree on
each extra $\bP^1$ is $1$ and the degree on one of the neighboring components is lowered by~$1$
(compared to the multidegree $\uone$).
Toric strata of codimension $k$ correspond to multidegrees $\hat d$ as above
that are semistable (i.e.,~satisfy the non-strict Gieseker's basic inequality \eqref{GBI} on $\hat\Sigma$).
The corresponding admissible sheaves on $\Sigma'$ are push-forwards of invertible sheaves $\hat F$ on $\hat\Sigma$
of a given admissible multidegree with respect to the stabilization morphism $\hat\Sigma\to\Sigma'$. 
\end{Review}

\begin{Proposition}
Suppose the hypergraph $\Gamma$ satisfies not only \eqref{CondS} but also \eqref{CondM}.
Boundary divisors of $\oPicone$ can be described as follows (using the language of \ref{CaporasoStuff}).
Either choose one of the $3$ nodes on any of the $n-6$ white components,
insert a $\bP^1$ at it, and lower the degree of the adjacent black component.
Or choose one of the $6$ nodes of $\Sigma$,  
insert a $\bP^1$ at it, and lower the degree of one of the $2$ adjacent black components.
In other words, the polytope of the toric variety $\oPicone$ has $3n-6$ faces.
\end{Proposition}

\begin{proof}
First of all, let us show that an admissible multidegree $\hat d$ on $\hat\Sigma$ is non-negative.
Indeed, let $Y\subset\hat\Sigma$ be one of the stable irreducible components
and suppose that $\deg\hat F|_Y<0$. Then \eqref{GBI} reads
$$
\left|\deg\hat F|_Y-{1\over 2n-8}(n-2)\right|\le {3\over2},
$$ 
which is nonsense.

Now we claim that all other possibilities give semistable multidegrees.
We will follow the argument of Prop.~\ref{stabledegree} and use the same notation.

Let $Y\subset\hat\Sigma$. Let $Z$ be a semistable irreducible component of $\hat\Sigma$.
Let $L$ be the ``lowered'' black component of $\hat\Sigma$.

We define $\eps$ as follows:
$$
\eps=\begin{cases}
1& \hbox{\rm if}\ Y\supset L,\ Y\not\supset Z,\cr
-1& \hbox{\rm if}\  Y\not\supset L,\ Y\supset Z,\cr
0&\hbox{\rm otherwise.}
\end{cases}
$$
By swapping $Y$ and $\hat\Sigma\setminus Y$ if necessary, we can assume that  
\begin{equation}\label{diamond1}
(n-2)w(Y)-(n-6)b(Y)+(2n-8)\eps\ge0.
\cooltag\end{equation}

The inequality  \eqref{GBI} is equivalent to
\begin{equation}\label{lucy0}
(n-2)w(Y)-(n-6)b(Y)+(2n-8)\eps-(n-4)\#(Y)\le0.
\cooltag\end{equation}

As in the proof of Prop.~\ref{stabledegree}, we can assume that all white components of $\hat\Sigma$
adjacent to a black component of $Y$ also belong to $Y$.

We can assume that $\eps\ne0$, because otherwise
inequality  \eqref{lucy0} is equivalent to  the same inequality \eqref{lucy} for $\Sigma^s$, since if we let $\ov{Y}$ be the image of $Y$ in
$\Sigma^s$, then $b(\ov{Y})=b(Y)$, $w(\ov{Y})=w(Y)$, $\#(\ov{Y})\geq\#(Y)$. 

If $\eps=-1$ then \eqref{lucy0} follows from \eqref{lucy} for $\Sigma^s$. 
So we assume that $\eps=1$.

If $Y$ contains the second component $W$ adjacent to $Z$ then again
\eqref{lucy0} follows from the same inequality \eqref{lucy} for $\Sigma^s$
(because $\#(\ov{Y})=\#(Y)+2$). So we can assume that $W\not\subset Y$.
Now consider $2$ cases.

Case A. Suppose that $W$ is white.
Then we have
$$w(Y)=e(Y)+c(Y)-r(Y)-1$$
and (by counting nodes in $Y$)
$$
3b(Y)-1=3(w(Y)-e(Y))+2r(Y)-(\#(Y)-1-3e(Y)),
$$
i.e.
$$
\#(Y)=-3b(Y)+3w(Y)+2r(Y)+2.
$$
Therefore, the LHS in \eqref{lucy0} becomes
$$
(-2n+10)e(Y)+(-2n+10)c(Y)-2r(Y)-(-2n+10)+(2n-6)b(Y)
$$
and so it suffices to prove that
$$
(2n-10)c(Y)+2r(Y)+(-2n+10)-(2n-6)b(Y)\ge0
$$
By \eqref{CondM}, this would follow from
$$
-4b+4n-20+2r\ge0,
$$
which follows from \eqref{finaleq}.

The Case B of a black component $W$
is similar and we leave it to the reader.
\end{proof}

\begin{Example}
The papers \cite{OS} and \cite{Al} contain a recipe for presenting
the polytope of $\oPicone$ as a slice of the hypercube. We won't go into the details here but let us give our favorite example.
Let $\Sigma$ be the Keel--Vermeire curve of Fig.~\ref{KHyper} with $4$ components
indexed by $\{1,2,3,4\}$.
Then the polytope is the rhombic dodecahedron of  Fig.~\ref{garnet}.
\begin{figure}[htbp]
\includegraphics[width=2in]{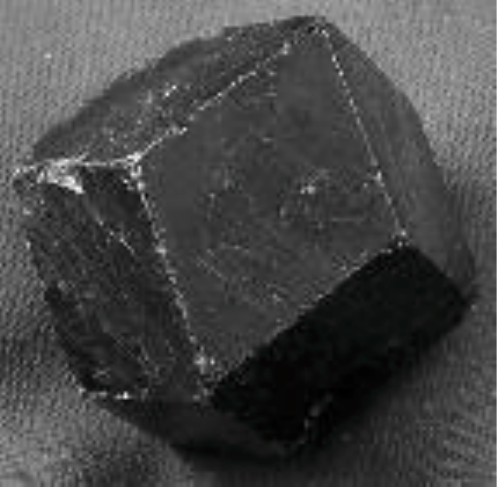}
\caption{\small Compactified Jacobian of the Keel--Vermeire curve.
}\label{garnet}
\end{figure}
The normals to its faces are given by roots $\alpha_{ij}=\{e_i-e_j\}$ of the root system $A_3$, where $i,j\in\{1,2,3,4\}$, 
$i\ne j$.
To describe a pure sheaf from the corresponding toric  codimension~$1$ stratum,
consider a quasi-stable curve $\Sigma_{ij}$
obtained by inserting a $\bP^1$ at the node of $\Sigma$ where $i$-th and $j$-th 
components intersect. Now just pushforward to $\Sigma$ an invertible sheaf
that has degree $1$ on this $\bP^1$ and at any component of $\Sigma_{ij}$ other than the proper transform
of the $i$-th component of $\Sigma$ (where the degree is $0$). 
\end{Example}

\begin{Remark}
What is the most natural compactification of $\Mor^{\uone}$ (and hence~$M_{0,n}$)
with respect to the morphism $v:\,\Mor^{\uone}\to\Pic^{\uone}$?
We note that $\Mor^{\uone}$ is a quotient of 
the space of morphisms $\Sigma^s\to\bP^1$ of multidegree $\uone$
by the free action of $\PGL_2$.
Let $\oM_{\Sigma^s}(\bP^1,n-2)$
be the Kontsevich moduli space 
(over $M_\Gamma=pt\in\oM_g$)
of stable maps $\hat\Sigma\to\bP^1$, where $\hat\Sigma$ is a semistable model of $\Sigma'$.
We would expect that the GIT-quotient of $\oM_{\Sigma^s}(\bP^1,n-2)$
by  $\PGL_2$ (with respect to some polarization) is the most obvious candidate. 
\end{Remark}

\section{Product of Linear Projections}\label{projections}

For a projective subspace $U\subset\bP^r$, let $l(U)=\codim U-1$
and let 
$$\pi_U:\,\PP^r\dra\PP^{l(U)}$$ 
be a linear projection from $U$.

\begin{Lemma}\label{linalg}
Let $U_1,\ldots,U_s\subset\PP^r$ be subspaces such that $U_i\not\subset U_j$ when $i\ne j$.
Then (a) the rational map 
$$\pi=\pi_{U_1}\times\ldots\times\pi_{U_s}:\,\PP^r\dra\PP^{l(U_1)}\times\ldots\times
\PP^{l(U_s)}$$ is dominant
if and only if 
\begin{equation}\label{dim inters}
l\left(\bigcap_{i\in S} U_i\right)\ge \sum_{i\in S}l(U_i)\quad\hbox{for any $S\subset \{1,\ldots,s\}$.}
\cooltag\end{equation}
(b) If $r=l(U_1)+\ldots+l(U_s)$ and $\pi$ is dominant then $\pi$ is birational.
\end{Lemma}

\begin{proof}
Let $l_i:=l(U_i)$.
The scheme-theoretic fibers of the morphism $\PP^r\setminus\bigcup_iU_i\to\PP^{l_1}\times\ldots\times\PP^{l_s}$
are open subsets of projective subspaces. This implies (b).

Assume $\pi$ is dominant and that (\ref{dim inters}) is not satisfied, for
example we may assume that $W=U_1\cap\ldots\cap U_m$ has dimension
$w\geq r-(l_1+\ldots+l_m)$.  The projections $\pi_{U_i}$ for $i=1,\ldots,m$ factor through
the projection $\pi_W:\,\PP^r\dra\PP^{r-w-1}$.
It follows that the map:
$$\pi'=\pi_{U_1}\times\ldots\times\pi_{U_m}:\,\PP^r\dra\PP^{l_1}\times\ldots\times\PP^{l_m}$$ factors through $\pi_W$.
If $\pi$ is dominant, then so is $\pi'$, and therefore
the induced map $\PP^{r-w-1}\dra\PP^{l_1}\times\ldots\times\PP^{l_m}$ is dominant, which contradicts
$w\geq r-(l_1+\ldots+l_m)$.

Assume (\ref{dim inters}). We'll show that $\pi$ is dominant. 
We argue by induction on $r$.
Let $H$ be a general hyperplane containing $U_s$.
It suffices to prove that the restriction of $\pi_{U_1}\times\ldots\times\pi_{U_{s-1}}$ on $H$
is dominant.
Subspaces $U'_i:=U_i\cap H$ have codimension~$l_i+1$ in $H$ and, therefore, by induction assumption,
it suffices to prove that 
\begin{equation}\label{dim intersprime}
\dim\bigcap_{i\in S} U'_i<(r-1)-\sum_{i\in S}l_i\quad\hbox{for any $S\subset\{1,\ldots,r-1\}$.}
\cooltag\end{equation}
Let $W:=\bigcap\limits_{i\in S} U_i$. Let $L:=\sum\limits_{i\in S}l_i$.
By   \eqref{dim inters}, $\dim W<r-L$ and, therefore,  
$\dim H\cap W<r-L-1$ (i.e.,~we have \eqref{dim intersprime})  
unless  $W\subset U_s$. But in the latter case
$\dim\bigcap\limits_{i\in S} U'_i=\dim(U_s\cap W)<r-(l_s+L)$ by   \eqref{dim inters}.
\end{proof}

We would like to work out the case when all subspaces $U_1,\ldots,U_s$ are intersections
of subspaces spanned by subsets of points $p_1,\ldots,p_n\in\bP^{n-2}$
in linearly general position.
Let $N=\{1,\ldots,n\}$.
For any non-empty subset $I\subset N$, let $H_I=\langle p_i\rangle_{i\not\in I}$.

\begin{Lemma}\label{crazystuff}
The rational map 
$$\pi=\pi_{H_{\Gamma_1}}\times\ldots\times\pi_{H_{\Gamma_l}}:\,\bP^{n-2}\dra\bP^{|\Gamma_1|-2}\times\ldots\times\bP^{|\Gamma_l|-2}$$
is dominant if and only if \eqref{CondS} holds.
\end{Lemma}

\begin{proof}
For any $S\subset\{1,\ldots,l\}$, let 
$e_S$ be the number of connected components of a hypergraph $\{\Gamma_i\}_{i\in S}$
with more than one element.
Let $\cH_S=\bigcap\limits_{i\in S}H_{\Gamma_i}$.

Let $W\subset\bA^{n}_{x_1,\ldots,x_{n}}$ be a hyperplane $\sum x_i=0$.
In appropriate coordinates,
$\bP(W)$ is a projective space dual to $\bP^{n-2}$ and subspaces $H_I\subset\bP^{n-2}$ 
are projectively dual to projectivizations of linear subspaces $\langle x_i-x_j\rangle_{i,j\in I}$.
It follows that $\cH_S$ is projectively dual 
to a subspace $\langle e_i-e_j\rangle_{\exists k\in S:\ i,j\in\Gamma_k}$, which implies 
that 
$$
l(\cH_S)=|\bigcup_{i\in S}\Gamma_i|-e_S-1.
$$
By~Lemma~\ref{linalg}, it follows that $\pi$ is dominant if and only if
\begin{equation}\label{CondC}
|\bigcup_{i\in S}\Gamma_i|-e_S-1\ge\sum_{i\in S}(|\Gamma_i|-2)
\quad\hbox{\rm for any $S\subset\{1,\ldots,l\}$}.
\cooltag\end{equation} 

It remains to check that \eqref{CondC} and \eqref{CondS} are equivalent. 
It is clear that \eqref{CondC} implies \eqref{CondS}.
Now assume  \eqref{CondS}.
Let $I_1,\ldots,I_{e_S}$ be connected components of $\Gamma_S$ 
with more than one  element.
This gives a partition $S=S_1\sqcup\ldots\sqcup S_{e_S}$ such that 
$I_k=\bigcup\limits_{j\in S_k}\Gamma_j$ for any~$k$. 
Applying  \eqref{CondS} for each $S_k$ gives
$$|\bigcup_{i\in S}\Gamma_i|-e_S-1\ge\sum_k\bigl(|\bigcup_{i\in S_k}\Gamma_i|-2\bigr)\ge\sum_k\sum_{i\in S_k}(|\Gamma_i|-2)=\sum_{i\in S}(|\Gamma_i|-2)$$
and this is nothing but \eqref{CondC}.
\end{proof}

For a subset $I\subset N$ with $|I|, |I^c|\ge2$,
let $\delta_I\subset\oM_{0,n}$ be the corresponding boundary divisor.
It is well-known that $\pi_{N}:\,\oM_{0,n+1}\to\oM_{0,n}$ 
is the universal family of~$\oM_{0,n}$ with sections
$\delta_{1,n+1},\ldots,\delta_{n,n+1}$.

\begin{lemdef}[\cite{Ka}]\label{BasicPsi}\label{diagram}\label{specialtriangles}
The line bundle 
$$\psi=\omega_{\pi_{N}}(\delta_{1,n+1}+\ldots+\delta_{n,n+1})$$
 is globally generated and gives the morphism 
 $$\Psi:\,\oM_{0,n+1}\to\bP^{n-2}$$
such that  $\Psi(\delta_{i,n+1})=p_i$ and, more generally, $\Psi(\delta_I)=H_{I}$ for any $I\subset N$.
$\Psi$ is an iterated blow-up of proper transforms of these subspaces 
in the order of increasing dimension. In particular,
$\Psi$ induces an isomorphism $M_{0,n+1}\simeq\bP^{n-2}\setminus\bigcup\limits_{i,j\in N}H_{ij}$.
We have a commutative diagram of maps ($\pi_{H_S}$ is a rational map):
$$
\begin{CD}
\oM_{0,n+1} @>{\Psi}>> \bP^{n-2}\\
@V{\pi_{S\cup\{n+1\}}}VV                 @VV{\pi_{H_S}}V\\
\oM_{0,k+1}  @>{\Psi}>> \bP^{k-2}\\
\end{CD}$$
for each subset $S\subset N$ with $k$ elements.
\end{lemdef}

\begin{proof}[Proof of Theorem~\ref{coolcondition}]
Follows from Lemma~\ref{crazystuff} and Lemma~\ref{specialtriangles}.
\end{proof}

\section{Proof of Theorems~\ref{NewDivisors} and~\ref{canonical}}\label{divisorssection}
By reordering, we can assume that $\Gamma_{n-3}=\{1,2,n\}$ and $\Gamma_{n-2}=\{3,4,n\}$.
We consider $\Gamma'$ as a hypergraph on $N':=\{1,\ldots,n-1\}$.
Consider a morphism
\begin{equation}\label{famousmapPi}
\Pi:=\pi_{\Gamma'\cup\{n\}}:\,\oM_{0,n}\ra(\oM_{0,4})^{n-4}=(\PP^1)^{n-4}.
\cooltag\end{equation}

\begin{Lemma}\label{morphismPi}
$\Pi$ is surjective and $\Pi|_{M_{0,n}}$  has irreducible $1$-dimensional fibers.
\end{Lemma}

\begin{proof}
This follows from Theorems~\ref{coolcondition} and \ref{ID}, respectively.
\end{proof}

Let $\tilde N:=\{a,b,5,6,\ldots,n\}$.
We identify a stratum $\delta_{12}\cap\delta_{34}\hra\M_{0,n}$ with
$\M_{0,\tilde N}$ by attaching to $a$ (resp.~to $b$) a rational tail with markings $1,2$ (resp.~$3,4$).
Let
$$\bD=\Pi(\delta_{12}\cap\delta_{34})\subset(\bP^1)^{n-4}.$$

\begin{Lemma}\label{birational}
$\bD$ is a divisor birational to $\delta_{12}\cap\delta_{34}$ via~$\Pi$
and $\cO(\bD)\simeq\cO(1,\ldots,1)$.
%
\end{Lemma}

\begin{proof} 
It suffices to prove that all compositions 
$$\oM_{0,\tilde N}\to(\bP^1)^{n-4}\to(\bP^1)^{n-5}$$
are birational, where the last map is one of the projections.
For any $\Gamma\subset N$, let $\tilde\Gamma\subset\tilde N$
be a subset obtained from $\Gamma$ by identifying $1,2$ with $a$ and $3,4$ with $b$.
By~Theorem~\ref{coolcondition}, we have to show that
$$
|\bigcup_{i\in S}\tilde\Gamma_i|\ge|S|+2
\quad\hbox{\rm for any proper subset $S\subset\{1,\ldots,n-4\}$}.
$$
But this follows from $(\dag)$ with only one potential exception:
if $S\subset\{1,\ldots,n-4\}$, $1<|S|<n-4$, and $1,2,3,4\in\bigcup_{i\in S}\Gamma_i$,
then we have to show that in fact $|\bigcup_{i\in S}\Gamma_i|\ge |S|+4$.
But suppose that $|\bigcup_{i\in S}\Gamma_i|<|S|+4$. Then  
$$|\Gamma_{n-3}\cup\Gamma_{n-2}\cup\bigcup_{i\in S}\Gamma_i|<|S|+5=
|S\cup\{n-3,n-2\}|+3,$$
which contradicts $(\dag)$.
\end{proof}

\begin{Proposition}\label{deltauv}
$\delta_{uv}\nsubseteq\Pi^{-1}(\bD)$ for any $u,v\in N'$.
If no triple $\Gamma_i$ ($i=1,\ldots,n-4$) contains both $u$ and~$v$ 
then $\Pi(\delta_{uv})=(\PP^1)^{n-4}$.
In particular,  $\delta_{12}$ and $\delta_{34}$ surject onto $(\bP^1)^{n-4}$.
\end{Proposition}

\begin{proof} 
Suppose $u,v\in N'$ and let $\hat N=\{p\}\cup N\setminus\{u,v\}$, $\hat N'=\hat N\cap N'$.
We identify $\delta_{u,v}$ with $\oM_{0,\hat N}$,
where $p$ is the attaching point. For $i=1,\ldots,n-4$, 
let $\hat\Gamma_i\subset\hat N'$ be a subset obtained from $\Gamma_i$
by identifying $u$ and $v$ with $p$. 
By \eqref{CondM}, we have 
\begin{equation}\label{jui}
|\bigcup_{i\in S}\hat\Gamma_i|\ge|S|+2
\quad\hbox{\rm for any $S\subset\{1,\ldots,n-4\}$ with $|S|\ge2$}.
\cooltag\end{equation}
 
Consider two cases. Suppose that there exists a triple, for instance $\Gamma_1$,
that contains both $u$ and $v$.
By \eqref{CondM} (with $|S|=2$), there is at most one triple with this property.
Then $\pi_{\Gamma_1\cup\{n\}}(\delta_{uv})$ is a point and 
$$\pi_{\Gamma_i\cup\{n\}}|_{\delta_{uv}}=\pi_{\hat\Gamma_i\cup\{n\}}\quad\hbox{\rm for $i>1$}.$$
By Theorem~\ref{coolcondition} and \eqref{jui}, it follows that
$$\Pi(\delta_{uv})=pt\times(\PP^1)^{n-5}\not\subseteq\bD,$$
because $\bD$ is a divisor of type $(1,\ldots,1)$.
In the second case, no triple $\Gamma_i$ contains both $u$ and $v$.
Then, arguing as above,
$\Pi(\delta_{uv})=(\PP^1)^{n-4}$.
\end{proof}

\begin{Proposition}\label{bdry chunk}
For any $\delta_S\subset\partial\M_{0,n}$  with $n\in S$, $1,2,3,4\notin S$ one has:
$$\Pi(\delta_S)=\Pi(\delta_S\cap\delta_{12}\cap\delta_{34}).$$
In particular, $\de_S$ is contained in $\Pi^{-1}(\bD)$.
\end{Proposition}

\begin{proof} 
One has:
$$\de_S=\M_{0,S\cup p}\times\M_{0,S^c\cup p},$$
$$\de_{\SS}=\de_S\cap\de_{12}\cap\de_{34}=\M_{0,\SS\cup p}\times\M_{0,\SS^c\cup p},$$
where we denote by $\T\subset\tN$ the set $T$ in which we identify $1=2=a$ and $3=4=b$. 
Note that under our
assumptions $\SS=S$.

The set $\{1,\ldots,n-4\}$ has a partition $A\cup B\cup C$ where:
$$i\in A\Leftrightarrow\Ga_i\subset S^c,\quad i\in B\Leftrightarrow|\Ga_i\cap S^c|=2,\quad i\in C\Leftrightarrow|\Ga_i\cap S^c|=0,1.$$
Let $\alpha=|A|$, $\beta=|B|$, $\gamma=|C|$. Then $\Pi=(\Pi_A,\Pi_B,\Pi_C)$ where
$$\Pi_A=(\pi_{\Ga_i\cup\{n\}})_{i\in A},\quad \Pi_B=(\pi_{\Ga_i\cup\{n\}})_{i\in B},
\quad \Pi_C=(\pi_{\Ga_i\cup\{n\}})_{i\in C}.$$

The morphism $\Pi_A|_{\de_S}:\de_S\ra(\PP^1)^{\alpha}$ factors through the morphism
$$\hat\Pi_A:\,\M_{0,S^c\cup\{p\}}\ra(\PP^1)^{\alpha}$$
given by the cross-ratios $\{\Ga_i\cup\{p\}\}_{i\in A}$.
$\hat\Pi_A$ is surjective by Lemma~\ref{crazystuff} and Lemma~\ref{specialtriangles}. Similarly,
the morphism $\Pi_A|{\de_{\SS}}:\,\de_{\SS}\ra(\PP^1)^{\alpha}$ factors through the map
$\M_{0,\SS^c\cup\{p\}}\ra(\PP^1)^{\alpha}$ given by the cross-ratios $\{\tilde\Ga_i\cup\{p\}\}_{i\in A}$
which is surjective, as Condition $(\dag)$ holds (the argument is identical to the one at the end of the proof of Lemma \ref{birational}).  
Hence, $\hat\Pi_A(\M_{0,S^c\cup\{p\}})=\hat\Pi_A(\M_{0,\SS^c\cup\{p\}})=(\bP^1)^{\alpha}$.

The map $\Pi_C|_{\de_S}:\,\de_S\ra(\PP^1)^{\gamma}$ factors through the map:
$$\hat\Pi_C:\,\M_{0,S\cup\{p\}}\ra(\PP^1)^{\gamma}$$
given by the cross-ratios $\{\Ga_i\cup\{n\}\}_{i\in C}$ after replacing any
$j\in S^c$ with $p$. Similarly, the map $\Pi_C|_{\de_{\SS}}:\,\de_{\SS}\ra(\PP^1)^{\gamma}$
factors through the map $\M_{0,\SS\cup\{p\}}\ra(\PP^1)^{\gamma}$ which is identical to the map $\hat\Pi_C$.
Hence,
$\hat\Pi_C(\M_{0,S\cup\{p\}})=\hat\Pi_C(\M_{0,\SS\cup\{p\}})$.

As $\Pi_B|_{\de_S}:\,\de_S\ra(\PP^1)^{\beta}$ is a constant map, 
it follows that $\Pi(\de_S)=\Pi(\de_{\SS})$.
\end{proof}

\begin{Lemma}\label{multip}
For $u>4$, $\delta_{u,n}$
is contained in 
$\Pi^{-1}(\bD)$ with multiplicity~$1$.
\end{Lemma}

\begin{proof}
Let $\rho$ be the restriction of $\Pi$ to $\delta_{34}\cong\M_{0,n-1}$.
Since no triple of $\Gamma'$ contains both $3$ and $4$, it follows from 
Lemma~\ref{deltauv} and Theorem \ref{coolcondition} that $\rho$ is birational.
Since $\delta_{12}\cap\delta_{34}$ is birational to~$\bD$
under $\Pi$, $\delta_{12}\cap\delta_{34}$ is the proper transform
of~$\bD$ under $\rho$:
\begin{equation}\label{DefF}
\rho^{-1}(\bD)=\delta_{12}\cap\delta_{34}+G,\quad\hbox{with}\quad G\subseteq\Exc(\rho),\quad \rho(G)\subset\bD.
\cooltag\end{equation}
Let $\De=\delta_{12n}+\delta_{34n}+\delta_{5n}+\ldots+\delta_{n-1,n}$
and let $\De'=\De|_{\delta_{34}}$.
Let $\delta'_{12}$ be $\delta_{12}\cap\delta_{34}$ as a boundary divisor of $\delta_{34}$. 
By Prop.~\ref{bdry chunk}, $\De\subset\Pi^{-1}(\bD)$.
If a component $\De'_0$ of $\De'$ is contained in
$\rho^{-1}(\bD)$ with multiplicity $\geq2$ then
$$\rho^{-1}(\bD)-\De'-\De'_0=\de'_{12}+G',$$
where $G'=G-\De'_0$ is $\rho$-exceptional. 

Let $\ga$ be the class of a general fiber $C$
of $\pi_{N'}|_{\delta_{12}\cap\delta_{34}}$, 
where $\pi_{N'}:\,\oM_{0,n}\to\oM_{0,n-1}$ is a forgetful map.
Then 
$$\ga.\de'_{12}=-1,\quad\ga.\de_{1,2,n}=\ga.\de_{3,4,n}=1,$$
$$\ga.\de_{i,n}=1,\quad i=5,6,\ldots,n-1.$$
(Here the first intersection is in $\de_{34}\cong\M_{0,n-1}$, while the rest are in $\M_{0,n}$.)

For any $i=1,\ldots,n-4$, 
the map $C\ra\PP^1$ given by the cross-ratio $\Ga_i\cup\{n\}$ is an isomorphism.
It follows that  
$\rho^{-1}(\bD).\ga=n-4$
by projection formula and therefore $(\rho^{-1}(\bD)-\De').\ga=-1$.
Since $\ga.\de'_{12}=-1$ ($\de'_{12}$ is the only boundary component in $\M_{0,n-1}$ that intersects $\gamma$ negatively) and since
$$(\rho^{-1}(\bD)-\De'-\De_0).\ga\leq-2,$$
it follows that $2\de'_{12}$ is a fixed part of
any effective divisor linearly equivalent to $\de'_{12}+G'$. In particular,
$G'-\de'_{12}$ is effective. This is a contradiction, as $G'$ is
$\rho$-exceptional, while $\de'_{12}$ is not.
\end{proof}

\begin{Notation}
Let $D=\ov{\Pi^{-1}(\bD)\cap M_{0,n}}$ 
\end{Notation}

\begin{Proposition}
The divisor $D$ is irreducible, non-empty, and surjects onto $\M_{0,n-1}$ via
$\pi:=\pi_{N\setminus\{i\}}:\,\M_{0,n}\ra\M_{0,n-1}$
for any $i\in N$.
\end{Proposition}

\begin{proof}
$D$ is irreducible by Lemma~\ref{morphismPi}.
Let $D^{main}$ be an irreducible component of $\Pi^{-1}(\bD)$ that contains $\de_{12}\cap\de_{34}\cong\oM_{0,n-2}$. 
By Lemma~\ref{deltauv},  $\delta_{12}$ and $\delta_{34}$ surject onto $(\bP^1)^{n-4}$.
It follows that $D^{main}$ intersects $M_{0,n}$ and therefore is equal to $D$.

If $i\in\{1,2\}$ (resp.~$i\in\{3,4\}$) then $\pi(\oM_{0,n-2})=\delta_{34}$
(resp.~$\delta_{12}$). Since $D$ intersects $M_{0,n}$, its image in $\M_{0,n-1}$
is not a subset of the boundary. Therefore, $\pi(D)=\oM_{0,n-1}$. 
Now we assume $i\not\in\{1,2,3,4\}$.

Arguing from contradiction, suppose $\pi(D)\neq\M_{0,n-1}$. 
Then $D.C=0$, where $C$ is a general fiber of $\pi$.
We prove this is not the case. 
Let 
$$I_i=\{k\in\{1,\ldots,n-4\}\,|\,i\in\Ga_k\cup\{n\}\}$$
and let $m_i=|I_i|$.
Then  
$$m_n=n-4\quad\hbox{\rm and}\quad m_i\ge2\quad \hbox{\rm if}\quad 4<i<n$$ 
because 
each $i\in N$ belongs to at least two triples from~$\Gamma$
(if $i$ belongs to at most one triple, say $\Gamma_k$, then 
$\bigl|\bigcup_{j\in\{1,\ldots,n-2\}\setminus\{k\}}\Ga_j\bigr|<n$
and this contradicts $(\dag)$).

It is clear that $\pi_{\Gamma_j\cup\{n\}}(C)$ is a point if $i\not\in\Gamma_j\cup\{n\}$ and 
$\pi_{\Gamma_j\cup\{n\}}:\,C\ra\PP^1$ is an isomorphism if $i\in\Gamma_j\cup\{n\}$. It follows (by projection formula) that 
$$\pi^{-1}(\bD).C=m_i.$$

Note that $\delta_{ij}$ are the only boundary divisors of $\oM_{0,n}$ that intersect $C$
(and the intersection number is~$1$).
By Lemma~\ref{deltauv}, $\delta_{uv}\nsubseteq\Pi^{-1}(\bD)$ for any $u,v\in N'$.
So it follows from Lemma~\ref{multip} that  
$D.C=m_n-(n-5)>0$ (if $i=n$) and $D.C=m_i-1>0$ (if $i\ne n$). 
\end{proof}

Theorem \ref{NewDivisors} now follows from the following:
\begin{Lemma}
$D$ is in the image of the exceptional divisor
of $\pi_{\Gamma\cup\{n+1\}}$.
\end{Lemma}

\begin{proof}
Indeed, take a general geometric point $z\in D\subset\oM_{0,n}\simeq\tilde G^1(\Sigma')$.
It is given by a morphism $f':\,\Sigma'\to\bP^2$, a pencil $|V|$ of lines in $\bP^2$,
and its member $H$. 
Points of the fiber of $\Pi:\,\tilde G^1(\Sigma')\to C^1(\Sigma')$ 
through $z$ correspond to the same $f'$ and $H$ but varying the pencil.
Since this fiber intersects $\delta_{12}\cap\delta_{34}$, it follows that the line $L_{12}$
connecting $1,2$ and the line $L_{34}$ 
connecting $3,4$ intersect at a point $P$ that belongs to $H$.
So now consider a morphism $f:\,\Sigma\to\bP^2$ which is equal to $f'$ on $\Sigma'$,
and which sends the remaining two components to $L_{12}$ and $L_{34}$
(and so the image of the $n$-th point is $P$).
Take a general line $\tilde H$ of $V$. Then $(f,|V|, \tilde H)$ is a point of 
$\tilde G^1(\Sigma)$ that maps to the given point of $\tilde G^1(\Sigma')$.
\end{proof}

We prove now Proposition \ref{Irred}. By Theorem \ref{ID} the exceptional locus of $\pi_{\Gamma\cup\{n+1\}}$ on $M_{0,n+1}$ is $\tilde{G}^2(\Sigma)$ 
and its image in $M_{0,n}$ is $G^2(\Sigma)$. Note, $\tilde{G}^2(\Sigma)$ has pure dimension 
$n-3$ (by van der W\"arden's purity theorem). As the canonical morphisms $\tilde{G}^2(\Sigma)\ra~G^2(\Sigma)$, $G^2(\Sigma)\ra W^2(\Sigma)$ have irreducible, equidimensional fibers, it follows that $G^2(\Sigma)$ (resp. $W^2(\Sigma)$) have pure dimension $n-4$ (resp. $n-6$). Moreover, irreducible components of 
$\tilde{G}^2(\Sigma)$ (resp. $G^2(\Sigma)$) are inverse images of irreducible components of $G^2(\Sigma)$ (resp. $W^2(\Sigma)$). Proposition \ref{Irred} is then a consequence of the following:

\begin{Lemma}\label{IrredLemma}
$W^2(\Sigma)$ is irreducible. In particular,  $G^2(\Sigma)$, $\tilde{G}^2(\Sigma)$ are irreducible.
\end{Lemma}
 
 \begin{proof} 
There is a well-defined morphism $\phi:W^2(\Sigma)\ra W^2(\Sigma')$ that maps an element $(\Sigma,L)$ of $G^2(\Sigma)$
to its restriction $(\Sigma',L|_{\Sigma'})$.  If $g:\Sigma\ra\PP^2$ is the map given by the complete linear system $|L|$, then 
$g|_{\Sigma'}:\Sigma'\ra\PP^2$ has non-degenerate image, hence $\dim\HH^0(\Sigma',L|_{\Sigma'})\geq2$. Clearly, $L|_{\Sigma'}$ is admissible.

We claim that the fiber of $\phi$ at a point $G^2(\Sigma')$ is either
(1) a point if $L_{12}\neq L_{34}$ and $P=L_{12}\cap L_{34}$ is not the image of a singular point of $\Sigma'$, or (2)
$1$-dimensional  if $L_{12}=L_{34}$ (and empty if $P=L_{12}\cap L_{34}$ is the image of a singular point of $\Sigma'$).  This is clear: in case (1) there is a unique way to extend a morphism $\Sigma'\ra\PP^2$ to a morphism $\Sigma\ra\PP^2$ by sending $n$ to $P$; in case (2), we send $n$ to a  point on the line $L_{12}=L_{34}$.

By Theorem \ref{ID} we have $\tilde{G}^1(\Sigma')=\tilde{G}^2(\Sigma')\cong M_{0,n}$. In particular, $\tilde{G}^2(\Sigma')$ is irreducible; hence, $W^2(\Sigma')$, $G^2(\Sigma'$ are irreducible.  Since $W^2(\Sigma)$ is pure-dimensional and since by assumption the locus in $W^2(\Sigma')$ where $L_{12}=L_{34}$ has codimension at least $2$, it follows that $W^2(\Sigma)$ is irreducible.
\end{proof}

\begin{Remark}
An extremal irreducible divisor $D$ on $\oM_{0,n}$ constructed in this section has an interesting
interpretation as a ``degeneracy'' divisor, which makes it similar, for example, to Eckhart divisors
on the Naruki moduli space of marked cubic surfaces \cite{HKT,CvG}
(and, by the way, we would conjecture that Eckhart divisors generate rays of the effective
cone of the Naruki space).
Namely, consider $M_{0,n}$ as the moduli space of pairs $(\PP^{n-3},p_1,\ldots,p_{n})$ of
$n$ points in linearly general position in $\PP^{n-3}$, up to automorphisms of $\PP^{n-3}$.
Let $u:\bX\ra M_{0,n}$ be the universal family of blow-ups:
the fiber of $u$ at the geometric point $(\PP^{n-3},p_1,\ldots,p_{n})$ is the
blow-up $\Bl_{p_1,\ldots,p_{n}}\PP^n$.
The effective cone of this blow-up does not depend on moduli and was computed in \cite{CT}.
But one can look at loci in the moduli space where the generators of the effective cone of the fiber
have unexpected intersection. For example, 
$(n-3)$-tuple from $\{1,\ldots,n\}$ uniquely determines a divisor
on $\Bl_{p_1,\ldots,p_{n}}\PP^{n-3}$: the proper transform of the hyperplane
in $\PP^{n-3}$ determined by the points in the $(n-3)$-tuple.
For simplicity,
we think of an $(n-3)$-tuple in terms of its complement in
$\{1,\ldots,n\}$, i.e.~a triple!
Now if we take  $n-2$ triples, i.e.~a hypergraph $\Gamma$ considered in this section,
then one expects that the corresponding
$n-2$ effective divisors in $\Bl_{p_1,\ldots,p_{n}}\PP^{n-3}$ do not intersect.
And we can define the locus in the moduli space where they do intersect.
It almost immediately follows from our calculations in this section
that this ``degeneracy'' divisor contains $D$ as an irreducible component.
\end{Remark}

Next we prove Theorem~\ref{canonical}.

\begin{Lemma}
$\omega_{\Sigma^s}$ is very ample.
\end{Lemma}

\begin{proof}
We use \cite[Prop.~2.5]{BE}:
it suffices to show that the removal of two edges from the dual graph $\Pi$ of $\Sigma^s$
does not disconnect it.
$\Pi$ has $n-2$ {\em black} vertices that correspond to components of~$\Sigma$
and $n-6$ {\em white} vertices that correspond to $\bP^1$'s inserted at triple points of~$\Sigma$. 
We argue by contradiction and assume that the removal of two edges does disconnect $\Pi$.
We can assume that each of the two ``halves'' of $\Pi$ contain at least two black vertices
(otherwise it is easy to get a contradiction). 

For $i=1,2$, let $w_i$ (resp.~$b_i$)
be the number of white (resp.~black) vertices in each half (hence, $b_1+b_2=n-2$, $w_1+w_2=n-6$).
Let $c_i$ be the number of singular points of $\Sigma$
covered by black components from the same half (i.e., using the dual graph, $c_i$ is the number of black-to-black edges in that half, 
plus the number of white vertices that are connected only to (black) vertices from the same half).
Note that $c_i\ge b_i+3$ by \eqref{CondM}. 
Let $r_i$ be the number of black--black edges in each component.
Consider several cases. 

Assume first that the two removed edges are both disconnecting. 
Suppose both disconnecting edges connect vertices of different color and the vertices 
in each half have different color.
Then we have 
$$w_i=c_i-1-r_i\ge b_i+2-r_i.$$
By adding these two inequalities (for $i=1,2$), we get $n-6\ge n-2+4-6$, a contradiction. (Note, $r_1+r_2=6$ in this case.)

Suppose both disconnecting edges connect vertices of different color and the vertices 
in each half have the same color.
Then we have
$$w_1=c_1-2-r_1\ge b_1+1-r_1,\quad w_2=c_2-r_2\ge b_2+3-r_2.$$
Suppose one disconnecting edge is black--black and another black--white. Then
$$w_1=c_1-2-r_1\ge b_1+1-r_1,\quad w_2=c_2-1-r_2\ge b_2+2-r_2.$$
Suppose both disconnecting edges are black--black. Then
$$w_1=c_1-2-r_1\ge b_1+1-r_1,\quad w_2=c_2-2-r_2\ge b_2+1-r_2.$$
In all cases adding these inequalities gives a contradiction
(note that in the last two cases $r_1+r_2=5$, respectively $r_1+r_2=4$).
The case when only one of the removed edges is disconnecting is similar, and we omit it.
\end{proof}

We now prove Theorem~\ref{canonical}. Since $W^2(\Sigma)$ has pure dimension $n-6$, 
Proposition \ref{Projections} implies that the closure of the image of $\psi$ is a component of 
$W^2(\Sigma)$.  Lemma \ref{IrredLemma} implies this is the only component of 
$W^2(\Sigma)$. 

\begin{Proposition}\label{Projections}
The map $\psi$ is injective on geometric points in its domain.
\end{Proposition}

\begin{proof}
By assumption, the domain of~$\psi$ is non-empty. 
Let $\Sigma_1,\ldots,\Sigma_{n-6}$ be the white components of $\Sigma^s$. Let $q_i\in\Sigma_i$ general, $Q=\sum q_i$.
Then $\psi(Q)$ is the map $\Sigma\ra\PP^2$ given by the complete linear system $|\omega_{\Sigma^s}(-Q)|$. 
If $R=\sum r_i\in(\PP^1)^{n-6}$ determines the same map, then $\omega_{\Sigma^s}(-Q)\cong\omega_{\Sigma^s}(-R)$.
However, we claim that $\h^0(\Sigma^s, \O(Q))=1$ (hence, $q_i=r_i$, for all $i$). Using Riemann-Roch and Serre duality
$$\h^0(\Sigma^s, \O(Q))-\h^0(\Sigma^s, \omega_{\Sigma^s}(-Q))=(n-6)+1-(n-3)=-2.$$
By assumption, 
$\omega_{\Sigma^s}(-Q)\in W^2(\Sigma)$ and $W^3(\Sigma)=\eset$. Hence, 
$$\h^0(\Sigma^s, \omega_{\Sigma^s}(-Q))=3$$ 
and the claim follows.
\end{proof}

Proposition \ref{ExtraCond} is a consequence of the following:
\begin{Lemma}\label{Hyp}
Let $L\in W^2(\Sigma)$ and let $H'$ be an admissible section of $L$. Consider the canonical embedding $\Sigma^s\hra\PP^{n-4}$.
There exists a unique hyperplane $H$ in $\PP^{n-4}$ containing $H'$. Moreover, $H$ does not contain any of the singular points of $\Sigma^s$.
\end{Lemma}

If $\Sigma_i$ is a white component of $\Sigma^s$ and let $q_i=H\cap\Sigma_i$, then
clearly, $$L=\omega_{\Sigma^s}(-q_1-\ldots-q_{n-6}),$$ i.e., the map induced by projections from $q_i$'s is in $W^2(\Sigma)$. In particular, the domain of $\psi$ is non-empty.

\begin{proof}[Proof of Lemma~\ref{Hyp}]
By Riemann-Roch and Serre duality, $$\h^0(\omega_{\Sigma^s}(-H'))-\h^0(H')=2(n-3)-2-(n-2)+1-(n-3)=-2.$$
Since $L\in W^2(\Sigma)$ and $W^3(\Sigma)=\eset$, we have $\h^0(H')=3$. It follows that $$\h^0(\omega_{\Sigma^s}(-H'))=1,$$ i.e., there is a unique hyperplane $H$ that contains $H'$. 

Assume that  $H$ contains a singular point of $\Sigma^s$ that lies on a black component $C$ corresponding to hyperedge $\Gamma_{\al}$. 
Since $H$ contains the non-singular point  $p=H'\cap C$, it follows that $C\subset H$, i.e., there is a section $s\neq0$ in $\H^0(\Sigma^s,\omega_{\Sigma^s})$ 
that vanishes along $C$.  Let $B$ be the curve obtained from $\Sigma^s$ by removing the component $C$. Then $s|_B\in\H^0(B,\omega_B)$. 
Let $\Phi$ be the curve obtained from $B$ by contracting the white components that intersect $C$ ($\Phi$ is the stable model of $\Sigma^{\alpha}$). Since the restriction of ${\omega_B}$ to any of the components of $\Sigma^s$ that intersect $C$ is trivial, it follows that $\omega_B$ is a pull-back of $\omega_{\Phi}$ . Hence,  $s|_B$ is the pull-back of a non-zero section in $\H^0(\Phi,\omega_{\Phi})$ that vanishes along $H'-p$. 
Note that $H'-p$ is an admissible section of $L|_{\Sigma^{\alpha}}\in W^2(\Sigma^{\alpha})$. By Riemann-Roch
$$\h^0(\Phi,H'-p)-h^0(\Phi,\omega_{\Phi}(-H'+p))=(n-3)+1-(n-5)=3,$$
(the genus of $\Phi$ is $n-5$). Since $\h^0(\Phi,\omega_{\Phi}(-H'+p))>0$, it follows that 
$$\h^0(\Phi,H'-p)=\h^0(\Sigma^{\alpha},L|_{\Sigma^{\alpha}})\geq4.$$ This is a contradiction, since by assumption 
$W^3(\Sigma^{\alpha})=\eset$.
\end{proof}

\begin{Review} 
Despite the fact that we proved extremality of hypergraph divisors on $\oM_{0,n}$
by a geometric argument, without ever computing their class, finding
it is an interesting combinatorial problem.
Under the assumptions of Theorem~\ref{NewDivisors},
one possibility to use the morphism \eqref{famousmapPi}
$$\Pi:\M_{0,n}\ra(\PP^1)^{n-4}.$$

For each boundary divisor $\de_S\subset\M_{0,n}$, let $m_S$ be the multiplicity with which 
$\de_S$ is contained in $\Pi^*{\bD}$. 
Then the class of $D$ is:
$$D=\Pi^*\O(1,\ldots,1)-\sum_S m_S\de_S.$$

Consider the Kapranov model of $\Psi:\M_{0,n}\ra\PP^{n-3}$:
$\Psi$ is an iterated blow-up of $\PP^{n-3}$ along $p_1,\ldots,p_{n-1}$ and proper transforms of subspaces $H_I$ spanned
by the points $p_i=\Psi(\de_{i,n})$ for $i\in I$, $I\subset N'$. 

Let $H=\Psi^*\O(1)$. Let $E_I$ be the exceptional divisors ($E_I=\de_{I\cup\{n\}}$). If $D=dH-\sum m_I E_I$, we call the coefficient $d$ of $H$ the \emph{H-degree} of $D$.  
 
 \begin{Proposition}\label{H-deg}
The $H$-degree of $D$ is $n-4$.
\end{Proposition}
\begin{proof}
The $H$-degree of $\Pi^*\O(1,\ldots,1)$ is $n-4$. The only boundary components that contribute to the $H$-degree are 
$\de_{uv}$, for $u,v\in N'$. Proposition \ref{H-deg} now follows from Proposition \ref{deltauv}.
\end{proof}
\end{Review}

The question of finding all the multiplicities $m_S$ is rather complicated in general, but in the case of the Keel-Vermeire divisor, it is easy to see that
the only boundary contained in $\Pi^{-1}(\bD)$ are $\de_{26}$, $\de_{146}$, $\de_{356}$ and we have:
$$D=\Pi^*\O(1,1)-\de_{26}-\de_{146}-\de_{356}=2H-\sum_{i=1}^5 E_i-E_{13}-E_{45}-E_{14}-E_{35}.$$
This is exactly how this divisor  was introduced be Vermeire \cite{V}.
A direct computation shows this class agrees with the Keel's description:
$D$ is the pull-back of the hyperelliptic divisor class on $\M_3$ 
via the morphism $\M_{0,6}\ra\M_3$
that send a $6$-pointed rational curve to a nodal curve of genus $3$ by identifying pairs of points $(14), (35), (26)$.

\section{Exceptional Curves on $\oM_{0,n}$}\label{CurvesSection}

Now we discuss exceptional loci of hypergraph morphisms that are not divisorial.

\begin{Proposition} 
An irreducible component of the exceptional locus
of the hypergraph morphism $\pi_{\Gamma\cup\{n+1\}}$ in \eqref{pgn} has dimension at least $3$
and the dimension of its image in $\oM_{0,n}$ is at least $2$. 
\end{Proposition}

\begin{proof}
Indeed, $\dim\tilde G^2=\dim C^2+1=\dim W^2+3$.
\end{proof}

To get an exceptional locus of dimension~$3$, we need a rigid hypergraph,
i.e.~such that $W^2$ is a point (or at least locally rigid, 
i.e.~$W^2$ has a zero-dimensional irreducible component) and $W^2\ne W^3$.
Let's give an example.
Consider the Hesse configuration \cite{AD} of $12$ lines $L_1,\ldots,L_{12}$ 
joining $9$ flexes $P_1,\ldots,P_9$ of a cubic in~$\bP^2$.
It gives a hypergraph $\Gamma$ with $12$ hyperedges (each is a triple) on $9$ vertices.

\begin{Proposition}\label{Hes}
 For a Hesse hypergraph $\Gamma$,
$W^2(\Gamma)$ is a point, the exceptional locus of $\pi_{\Gamma\cup\{10\}}$ is a threefold $T$,
and Theorem~\ref{ID} induces a commutative diagram
$$
\begin{CD}
\Bl_{L_1,\ldots,L_{12}}(\bP^2)^\vee @<<< T @>>> \Bl_{P_1,\ldots,P_9}\bP^2 @>>> W^2(\Gamma)=pt\\
@VVV                            @VVV                          @VVV @VVV \\
(\bP^1)^{12} @<\pi_{\Gamma\cup\{10\}}<< \oM_{0,10} @>\pi>> 
\oM_{0,9}\hbox to 0pt{$\quad\quad\quad\mathop{\dra}\limits^v$} @.  \Pic^{\uone}=\bG_m^{16}\\
\end{CD}
$$
where vertical arrows are closed immersions and $v$ is a rational map regular on~$M_{0,9}$.
\end{Proposition}

\begin{proof}
Let us recall how to prove a well-known fact that $\Gamma$ is rigid,
i.e.~$W^2$ is a point (see \cite{U} for a different and much more general approach
close in spirit to our calculations in Section~\ref{projections}).
Firstly, one can prove that for any $L\in W^2$, 
$|L|$ maps different components of $\Sigma$ to 
different lines (otherwise the combinatorics of $\Gamma$ forces all components to have the same image).
Note that the $9$ points are separated by $|L|$.
Consider the linear system of cubics through the $9$ points.
It contains $4$ ``triangles'' made out of the $12$ lines.
Therefore this linear system has no fixed components
and so it must be a pencil (because $8$ points in $\bP^2$ impose independent
conditions on cubics unless five of them lie on a line
or all eight lie on a conic, but in both cases the linear
system of cubics through $9$ points would have fixed components).
To show that this is the Hesse pencil, it suffices to show that it is stable under taking the Hessian.
Suppose our pencil is $|aF+bG|$. The Hessian is a cubic form
that depends on $a$ and $b$. Consider the vector 
$$v(a,b)=\Hes(aF+bG)\mod\langle F,G\rangle.$$
It has degree $3$ in $a$ and $b$
and so it suffices to show that it is trivial for $4$ different values of $(a:b)$.
But the Hessian of $xyz$ is $2xyz$ and so $v(a,b)=0$ for values
of $(a:b)$ that correspond to $4$ triangles.

So in this case the exceptional locus is a threefold $T$.
It follows from Prop.~\ref{blowupdescr} below that
the closure of $U$ in $\oM_{0,n}$
is isomorphic to $\Bl_{p_1,\ldots,p_n}\bP^2$, where 
the points $p_1,\ldots,p_n\in\bP^2$ are the images of singular points of $\Sigma$.
\end{proof}

\begin{Proposition}\label{blowupdescr}
Suppose $p_1,\ldots,p_n\in\bP^2$ are distinct points, and 
let $U\subset\bP^2$ be the complement to the union of lines connecting them.
The morphism 
$$F:\,U\to M_{0,n}$$ 
obtained by projecting $p_1,\ldots,p_n$ from points of $U$
extends to the morphism 
\begin{equation}\label{cutemorphism}
F:\,\Bl_{p_1,\ldots,p_n}\bP^2\to\oM_{0,n}.
\cooltag\end{equation}

If there is no (probably reducible) conic through  $p_1,\ldots,p_n$
then $F$ is a closed embedding.
In this case the boundary divisors $\delta_I$ of $\oM_{0,n}$ pull-back as follows:
for each line $L_I:=\langle p_i\rangle_{i\in I}\subset\bP^2$, we have $F^*(\delta_I)=\tilde L_I$ 
(the proper transform of $L_i$) and (assuming $|I|\ge3$),
$F^*(\delta_{I\setminus\{k\}})=E_k$, where $k\in I$ and $E_k$ is the exceptional divisor over $p_k$.
Other boundary divisors pull-back trivially.
\end{Proposition}

\begin{proof}
For any $I\subset\{1,\ldots,n\}$, we denote by $F_I:\,\bP^2\dra \oM_{0,I}$ a
rational map defined as above but using only points $p_i$, $i\in I$. Then $F_I=\pi_I\circ F$, for the forgetful map
$\pi_I:\M_{0,n}\ra\M_{0,I}$.

First suppose that $n=4$. Consider three cases.
If no three out of the four points $p_1,\ldots,p_4$ lie on a line then 
$$F:\,\Bl_{p_1,p_2,p_3,p_4}\bP^2\simeq\oM_{0,5}\to\oM_{0,4}\simeq\bP^1$$ 
is  given by the pencil of conics through $p_1,\ldots,p_4$.
If $p_1,p_2,p_3$ lie on a line that does not contain $p_4$ then 
$F:\,\Bl_{p_4}\bP^2\to\oM_{0,4}\simeq\bP^1$ is a projection from $p_4$.
Finally, if all points lie on a line then $F_{1234}$ is a map to a point (given by the cross-ratio of $p_1,\ldots,p_4$
on the line they span).  Note that in all cases $F$ is regular on $\Bl_{p_1,\ldots,p_n}\bP^2$.
The product of all forgetful maps $\oM_{0,n}\to\prod_I\oM_{0,I}$ over all $4$-element subsets
is a closed embedding (see e.g.~\cite[Th.~9.18]{HKT}).
It follows that \eqref{cutemorphism} is regular.

Now suppose that there is no conic passing through all points.

The argument above shows that $F$ restricted to each exceptional divisor $E_i$ is a closed immersion.
Indeed,  there always exist three points $p_a,p_b,p_c$ such that $p_i$ does not belong
to a line spanned by any two of the three points (otherwise all points belong to a union of two lines
passing through $p_i$, which is a reducible conic). 
By the above, the morphism $F_{abci}|_{E_i}$ is a closed immersion (in fact an isomorphism).

Let $k$ be the maximal number such that there exist $k$ points out of $p_1,\ldots,p_n$
lying on a smooth conic. We can assume without loss of generality that $p_1,\ldots,p_k$ lie on smooth conic.
We consider several cases. Suppose first that $k\ge5$.
Since, for any $4$-element subset $I\subset\{1,\ldots,k\}$, $F_I$ is given by a linear system of conics 
through $p_i$, $i\in I$, the geometric fibers of $F_{1\ldots k}$ are:
(1) the proper transform $\tilde C$ of a conic $C$ through $p_1,\ldots,p_k$ (which does not pass through the remaining points);
(2) exceptional divisors $E_i$, $i>k$; (3) closed points in $\bP^2\setminus\{p_1,\ldots,p_n\}$.
Since we already know that $F|_{E_i}$ is a closed embedding,
it suffices to prove that $F|_{\tilde C}$ is a closed embedding. 
For this, consider $F_{123,k+1}$. There are two subcases.
If $p_1,p_2,p_3,p_{k+1}$ lie on a smooth conic then, since $p_{k+1}\not\in C$, the linear system of
conics through $p_1,p_2,p_3,p_{k+1}$ separate points of $\tilde C$.
If they lie on a reducible conic then $p_{k+1}$ must belong to a line connecting
a pair of points from $p_1,p_2,p_3$, for example $p_2$ and $p_3$.
Then the linear system of lines through $p_1$ separate points of~$\tilde C$.
In both cases, $F|_{\tilde C}$ is a closed embedding.

Note that $k\ne 2$ (otherwise all points lie on a line through $p_1$ and $p_2$).
We claim that $k\ne 3$ either. Arguing by contradiction, suppose that $k=3$.
Then, for any $i>3$, $p_i$ lies on one of the three lines connecting $p_1$, $p_2$, $p_3$.
Moreover, each of these lines must contain at least one of the points $p_i$, $i>3$,
because otherwise all points lie on the union of two lines.
So suppose that 
$$
p_4\in\langle p_1,p_2\rangle,\quad
p_5\in\langle p_2,p_3\rangle,\quad
p_6\in\langle p_1,p_3\rangle.$$ 
But then $p_2,p_3,p_4,p_6$ lie on a smooth conic.

So the only case left is $k=4$. 
Points $p_5,\ldots,p_n$ lie on a union of $6$ lines
connecting $p_1,\ldots,p_4$ pairwise. 
The geometric fibers of $F_{1234}:\,\Bl_{p_1,\ldots,p_n}\bP^2\to\oM_{0,\{1,2,3,4\}}$
are the preimages w.r.t.~the morphism $\Bl_{p_1,\ldots,p_n}\bP^2\to\Bl_{p_1,\ldots,p_4}\bP^2$
of proper transforms of conics $C$ through $p_1,\ldots,p_4$.
If $C$ is a smooth conic then the argument from the $k\ge5$ case
shows that $F|_{\tilde C}$ is a closed embedding.
So suppose that $C$ is a reducible conic, for example the union of lines $\langle p_1,p_2\rangle$
and $\langle p_3,p_4\rangle$. 
Note that not all points belong to these two lines, for example suppose $p_5$ belongs to $\langle p_1,p_3\rangle$.
Then $F_{1352}$ collapses $\langle p_1,p_2\rangle$ and separates points of $\langle p_3,p_4\rangle$.
$F_{1354}$ has an opposite effect. So $F_{13524}$ separates points of $\tilde C$ and we are done.

To compute pull-backs of boundary divisors, note that $F^{-1}(\partial\oM_{0,n})=\partial U$
(set-theoretically), and so, for any subset $I$, $F^*\delta_I$ (as a Cartier divisor) is a linear combination
of proper transforms of lines $L_J=\langle p_j\rangle_{j\in J}$ and exceptional divisors~$E_i$.
In order to compute multiplicity of $F^*\delta_I$ at one of these divisors $D$,
we can argue as follows: suppose $C\subset\Bl_{p_1,\ldots,p_n}\bP^2$
is a proper curve intersecting $D$ transversally at a point $p\in C$ that does not belong to any other boundary component.
By the projection formula, the multiplicity is equal to the local intersection number of $F(C)$ with 
$\delta_I$ at $F(p)$. But this intersection number can be immediately computed 
from the pullback of the universal family of $\oM_{0,n}$ to $C$.
To implement this program, we consider two cases.
First, suppose that $D=L_J$.
Working locally on $\bA^2_{x,y}\subset\bP^2$, we can assume that $p=(x,y)$,
$D=(x)$, $C=(y)$, $J=\{1,\ldots,k\}$, $p_i=(x,y-b_i)$, $b_i\ne0$, for $i\le k$, and $p_i=(x-a_i,y-b_i)$, for $i>k$, where 
$a_i\ne0$, $b_i\ne 0$, and $a_i/b_i\ne a_j/b_j$. Then (locally near $p$) the pull-back of the universal family
of $\oM_{0,n}$ to the punctured neighborhood $U\subset C$ of $p$ has a chart $\Spec k[x,1/x,s]_{(x)}\to\Spec k[x,1/x]_{(x)}$
with sections $(x+sb_i)$ for $i\le k$ and $(x+sb_i-a_i)$ for $i>k$.
Closing up the family in $\Spec k[x,s]_{(x)}\to\Spec k[x]_{(x)}$ and blowing-up the origin
$(x,s)\in \Spec k[x,s]_{(x)}$
separates the first $k$ sections. The special fiber has two components,
with points marked by $J$ one component and points marked by $J^c$ on the other.
This proves the claim in the first case.

Secondly, suppose that $D=E_1$. We assume that $p_1=(x,y)\in\bA^2\subset\bP^2$.
We work on the chart $\Spec k[x,t]\subset\Bl_{p_1}\bA^2$ where $y=tx$.
Then $E_1=(x)$. We can assume that $p=(x,t)$, $C=(t)$, and that $p_i=(x-a_i, t-t_i)$ for $i>1$, where $a_i\ne0$, $t_i\ne0$.
Then (locally near $p$) the pull-back of the universal family
of $\oM_{0,n}$ to the punctured neighborhood $U\subset C$ of $p$ has a chart $\Spec k[x,1/x,s]_{(x)}\to\Spec k[x,1/x]_{(x)}$
with sections $s_1=(s)$, $s_i=(s-t_i-sxa_i^{-1})$ for $i>1$. 
We close-up in $\Spec k[x,s]_{(x)}\to\Spec k[x]_{(x)}$
and resolve the special fiber by blowing up points $(x,s-t_i)$ each time there is more than one point 
with the same slope $t_i$. This yields a family of stable curves with a special fiber that contains
(a) a ``main'' component with points marked by $1$ and by $i$ each time 
there is just one point with the slope~$t_i$; (b) one component (attached to the main component)
for each $t_i$ that repeats more than once marked by $j$ such that $t_i=t_j$.
This proves the claim in the second case.
\end{proof}

\begin{Example}
Applying this to $n=6$ gives a covering of $\oM_{0,6}$ by cubic surfaces. This is related to the fact that $\oM_{0,6}$
is a resolution of singularities of the Segre cubic threefold
$$\cS=\{(x_0:\ldots:x_5)\,|\,\sum x_i=\sum x_i^3=0\}\subset\bP^5.$$
It is easy to check that our blow-ups are pull-backs of hyperplane sections of~$\cS$.
This proves the well-known fact that moduli of cubic surfaces are generated by hyperplane sections of $\cS$
(the Cremona hexahedral equations, see \cite{Do}).
It deserves mentioning that one of the (non-general) blow-ups of $\bP^2$ in $6$ points 
embedded in $\oM_{0,6}$ this way is the Keel--Vermeire divisor, 
see Example~\ref{KeelVermeire}.
\end{Example}

In our quest for even smaller exceptional loci let's first look at the morphism
$$M_{0,n}\simeq G^1 \arrow^v W^1$$
of Theorem~\ref{ID}. 
We have the following corollary of Theorem~\ref{ID}:

\begin{Corollary}\label{Cor1}
(a) The exceptional locus of $v$ is equal to $G^2$.
Suppose that $\Sigma/k$ is a hypergraph curve and that $L\in W^2(\Sigma)\setminus W^3(\Sigma)$
is an admissible line bundle giving a morphism
$$f:\,\Sigma\to\bP^2.$$
Let $U:=\bP^2\setminus f(\Sigma)$.
The geometric fiber of $v:\,G^1 \to W^1$ over $(\Sigma,L)\in W^2$
is isomorphic to~$U$.
Its geometric points correspond to morphisms 
$$\Sigma\to f(\Sigma)\arrow^{pr_x}\bP^1,$$ 
where $pr_x:\,\bP^2\dra\bP^1$ is a linear projection from $x\in U$.

(b) If $W^2$ is a point (and $W^3$ is empty) then $U$ is the exceptional locus of $v$.
\end{Corollary}

The problem is that $M_{0,n}\arrow^v W^1$ does not induce the morphism
from~$\oM_{0,n}$. We explored this issue in detail in Section~\ref{smaps}.

\

Our last hope is the hypergraph morphism
$$\pi_\Gamma:\,M_{0,n}=G^1\to M_\Gamma=\prod\limits_{j=1}^d M_{0,\Gamma_j}.$$ 
It extends to a morphism from $\oM_{0,n}$ but it is much trickier to study then $\pi_{\Gamma\cup\{n+1\}}$.
Let  $p:\,W^1\to M_\Gamma$ be the structure map. From now on we assume that we have
the setup of Cor.~\ref{Cor1}~(b), i.e., ~$W^2$ is a point and $W^3$ is empty.
Let $m_0=p(W^2)$ and let $\Sigma$ be the fiber of the universal family of hypergraph curves
over $m_0\in M_\Gamma$.
Let $p_1,\ldots,p_n\in\bP^2_k$ be the images of singular points of $\Sigma$ under 
the linear system $|W^2|$.

\begin{Proposition}\label{propermap}
In the setup of Cor.~\ref{Cor1} (b), 
$U$ belongs to the exceptional locus of $\pi_\Gamma:\,M_{0,n}\to M_\Gamma$.
If, moreover,  points  $p_1,\ldots,p_5$ lie on a smooth conic~$C$, then
$C\cap U$ belongs to the exceptional locus of the morphism 
\begin{equation}\label{record}
\pi_\Gamma\times\pi_I:\,M_{0,n}\arrow M_\Gamma\times M_{0, I},
\cooltag\end{equation}
where $I=\{1,2,3,4,5\}$.
If $C\cap U$ is an irreducible component of the exceptional locus then $\overline{C\cap U}\subset\oM_{0,n}$ 
is a rigid curve on $\oM_{0,n}$.
\end{Proposition}

\begin{proof}
$C\cap U$ is clearly the exceptional locus for the map $U\to M_{0,5}$
given by projecting $p_1,\ldots,p_5$ from points of $U$. 
Rigidity of $\overline{C\cap U}$ follows from  Mumford's rigidity lemma \cite[1.6]{KoM}.
\end{proof}

It remains to find a hypergraph that satisfies the last condition of Prop.~\ref{propermap}.
At the very least we need $\Gamma$ such that $W^1$ has relative dimension~$0$.
By the Brill--Noether theory, the relative dimension of $W^1$ is at least 
$$g-2(g-d+1)=\dim M_{0,n}-\dim M_\Gamma.$$

\begin{Review}
Consider the hypergraph of the {\em dual}\/ Hesse configuration.
We use the following enumeration of its hyperedges:
 $$\Ga_1=\{p, 1,b,\ga\},\quad \Ga_2=\{p,2,c,\be\},\quad \Ga_3=\{p,3,a,\al\}$$
$$\Ga_4=\{n, 2,a,\ga\},\quad \Ga_5=\{n,3,b,\be\},\quad\Ga_6=\{n,1,c,\al\}$$
$$\Ga_7=\{m, 1,2,3\},\quad \Ga_8=\{m,\al,\be,\ga\},\quad  \Ga_9=\{m,a,b,c\}$$
It has $d=9$ hyperedges with $4$ points on each hyperedge, with $12$ vertices.
Note that $g=16$ and the expected relative dimension of $W^1$ is $0$.

Let $\Ga$ be the hypergraph $\{\Ga_1,\ldots,\Ga_9,\Ga_0\}$ where:
$$\Ga_0=\{m,n,p,1,a\}$$
(this corresponds to adding a conic $C$ through $5$ points in Prop.~\ref{propermap}).
\end{Review}

\begin{Theorem}\label{HesseMain} 
A hypergraph morphism  $$\pi:M_{0,12}\ra M_{\Ga}=(M_{0,4})^9\times M_{0,5}$$ 
has a $1$-dimensional connected component in the closure of its exceptional locus in $M_{0,12}$.
This connected component is in fact irreducible and is 
the proper transform $C$ in $\Bl_{12}\bP^2$ of the conic in $\bP^2$
passing through $5$ points $\{m,n,p,1,a\}$ of the dual Hesse configuration.\end{Theorem}

\begin{proof}
Let $\rho$ be a closed point of $M_{0,12}=G^1(\Sigma)$. Then $\rho$ gives rise to the morphism $\Sigma\to\bP^1$
and we let $x'=\rho(x)$ for any singular point $x$ of $\Sigma$.
Without loss of generality we can assume that 
$$1'=\infty,\quad m'=0,\quad a'=1,$$
and we let
$$b'=t,$$
where $t\in k$ is a parameter.

In this coordinates the hypergraph morphism  
$\pi:M_{0,12}\ra M_{\Ga}=(M_{0,4})^9\times M_{0,5}$ has the following form:
$$w_1=\{p', 1',b',\ga'\},\quad w_2=\{p',2',c',\be'\},\quad w_3=\{p',3',a',\al'\},$$
$$w_4=\{n', \ga',a',2'\},\quad w_5=\{n',\be',b',3'\},\quad w_6=\{n',\al',c',1'\},$$
$$w_7=\{m', 1',2',3'\},\quad w_8=\{m',\ga',\be',\al'\},\quad  w_9=\{m',b',c',a'\},$$
$$u=p',\quad v=n',$$
where $$\{x,y,z,s\}=\frac{(s-x)(y-z)}{(y-x)(s-z)}$$ is the cross-ratio
and $(u,v)$ are coordinates on $M_{0,5}$.

\begin{Claim}
The natural morphism $M_{0,12}\to (M_{0,4})^9\times M_{0,\{1,m,a,b,p,n\}}$
is injective on closed points. In particular, $\pi_\Gamma$ has at most one-dimensional fibers.
\end{Claim}

\begin{proof}
We will show how to recover all points $x'$ starting from $1'$, $m'$, $a'$, $b'$, $p'$, $n'$
and using coordinates on $M_\Gamma$.  From the cross-ratio $w_9$ we find that:
$$c'=\frac{(w_9-1)t}{w_9t-1}.$$
From the cross-ratio $w_1$ we find that:
$$\ga'=\frac{w_1t-u}{w_1-1}.$$
From the cross-ratio $w_4$ we find that:
$$2'=\frac{-w_4v+v+\ga'(w_4-v)}{-w_4v+1+\ga'(w_4-1)}=\frac{-v(w_4-1)(w_1-1)+(w_4-v)(w_1t-u)}{(1-w_4v)(w_1-1)+(w_4-1)(w_1t-u)}.$$
For simplicity, we think of this as $2'=\frac{C}{D}$ where
\begin{equation}\label{C}
C=-v(w_4-1)(w_1-1)+(w_4-v)(w_1t-u),
\cooltag\end{equation}
\begin{equation}\label{D}
D=(1-w_4v)(w_1-1)+(w_4-1)(w_1t-u).
\cooltag\end{equation}
From the cross-ratio $w_6$ we find that:
$$\al'=\frac{w_6v-c'}{w_6-1}=\frac{w_6v(w_9t-1)-(w_9-1)t}{(w_6-1)(w_9t-1)}.$$
For simplicity, we think of this as $\al'=\frac{A}{B}$ where
\begin{equation}\label{A}
A=w_6v(w_9t-1)-(w_9-1)t,
\cooltag\end{equation}
\begin{equation}\label{B}
B=(w_6-1)(w_9t-1).
\cooltag\end{equation}
From the cross-ratio $w_7$ we find that:
$$3'=\frac{w_72'}{w_7-1}=\frac{w_7C}{(w_7-1)D}.$$
Finally, from the cross-ratio $w_8$ we find that:
$$\be'=\frac{M}{N}$$ where we denote:
\begin{equation}\label{M}
M=(1-w_8)(w_1t-u)A,
\cooltag\end{equation}
\begin{equation}\label{N}
N=(w_1-1)A-w_8(w_1t-u)B.
\cooltag\end{equation}
This shows the claim.
\end{proof}

\begin{Lemma}
The locus in $M_{\Ga}$ where the fiber of the hypergraph map is positive dimensional 
is given by those points for which the following polynomials in $t$
with coefficients in $k[w_1,\ldots,w_9,u,v]$  
are identically zero:
\begin{equation}\label{C3}
(A-uB)[w_7C-(w_7-1)D]-w_3(A-B)[w_7C-u(w_7-1)D]=\\
f_1t^2+f_2t+f_3,
\cooltag\end{equation}

$$
[w_7C-v(w_7-1)D](M-tN)-w_5(M-vN)[w_7C-t(w_7-1)D]\\
$$
\begin{equation}\label{C5}
=f_4t^4+f_5t^3+\ldots+f_8,
\cooltag\end{equation}

$$
[(w_9t-1)C-(w_9-1)tD](M-uN)-w_2[(w_9t-1)M-(w_9-1)tN](C-uD)
$$\begin{equation}\label{C2}
=f_9t^4+f_{10}t^3+\ldots+f_{13},
\cooltag\end{equation}
where $A,B,C,D,M,N$ are as in \eqref{C} -- \eqref{N}.
\end{Lemma}

\begin{proof}
We get equations on $t$ by utilizing the cross-ratios not used in the proof of the previous Claim.
Namely, we get \eqref{C3} from the points $3'$, $p'$, $a'$, $\al'$ and $w_3$;
we get \eqref{C5} from the points $n'$, $3'$, $b'$, $\be'$ and $w_5$;
we get \eqref{C2} from the points $p'$, $2'$, $c'$, $\be'$ and $w_2$.
\end{proof}

Let $m_0\in M_\Gamma$ be a point that corresponds to the dual Hesse configuration in $\bP^2$.
It is not realizable over $\bR$, so we can give 
only its vague sketch, see Fig.~\ref{dHesse}.

\begin{figure}[htbp]
\includegraphics[width=4in]{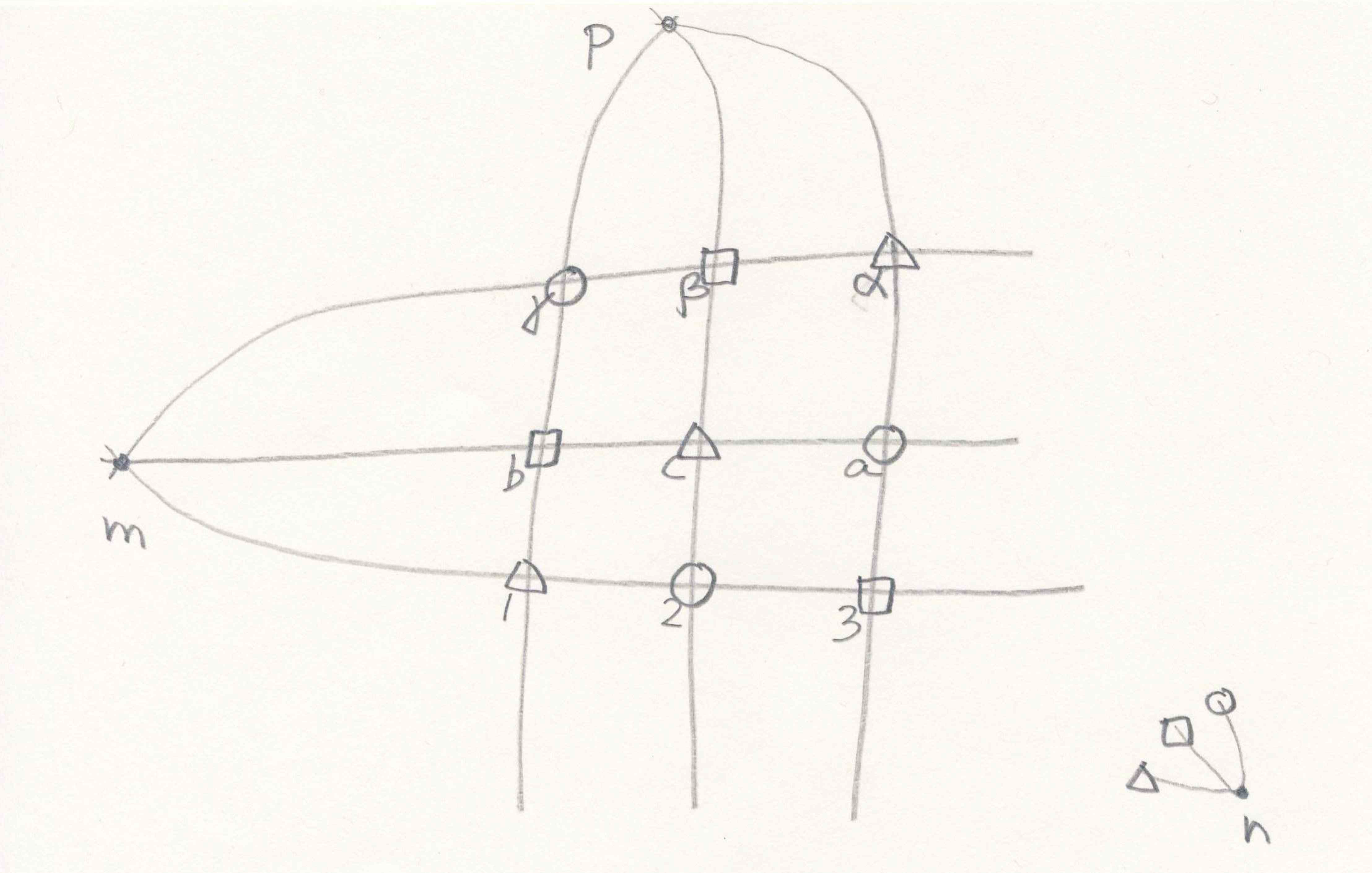}
\caption{\small A dual Hesse hypergraph.
}\label{dHesse}
\end{figure}
 
Note that ``circles'' (resp.~``squares'', resp.~``triangles'') span lines $\Ga_4$, $\Ga_5$, and $\Ga_6$.
Alternatively, one can choose coordinates in $\bP^2$ such that
$$\Ga_1\Ga_2\Ga_3=X^3-Y^3,\quad  \Ga_4\Ga_5\Ga_6=Y^3-Z^3,
\quad \Ga_7\Ga_8\Ga_9=Z^3-X^3.$$

\begin{Lemma} Let $\omega$ be the primitive cubic root of~$1$.
$m_0$ has coordinates 
$$w_1^0=\ldots=w_9^0=-\omega^2,\quad  u^0=1-\omega,\quad v^0=1-\omega^2.$$
The differentials of functions $f_1,\ldots,f_{13}$ at $m_0$ do not depend on $u$ and $v$
and the Jacobian matrix $[\partial f_i/\partial w_j]$ at $m_0$
is given by  
$$\left[\small\begin{array}{ccccccccccccccccccccccc}
0& 0& 0& 1& 0& \omega+1& -\omega-1& 0& \omega+1\cr
-1& 0& \omega& \omega& 0& -\omega-3& 2 \omega+3& 0& -\omega-1\cr
-\omega+1& 0& 0& 0& 0& -\omega+1& -\omega-2& 0& 0\cr
0& 0& 0& 1& -\omega-1& \omega& -\omega-1& 1& \omega\cr
0& 0& 0& 2 \omega-1& 3 \omega+4& -5 \omega-3& 3 \omega+5& 3 \omega-1& -3 \omega-1\cr
0& 0& 0& -5 \omega-1& -3 \omega-9& 8 \omega+10& -2 \omega-10& -9 \omega-3& 5 \omega+4\cr
0& 0& 0& 3 \omega+3& 9& -3 \omega-12& -3 \omega+9& 9 \omega+9& -3 \omega-3\cr
0& 0& 0& 0& 3 \omega-3& -3 \omega+3& 3 \omega-3& -3 \omega-6& 0\cr
0& 0& 0& -2 \omega& 0& 2& 0& -2 \omega-2& 2 \omega+2\cr
-2& 4 \omega+4& 0& 6 \omega+7& 0& 9 \omega-1& 0& -\omega+8& -\omega-6\cr
-7 \omega+1& -12& 0& \omega-7& 0& -20 \omega-16& 0& 15 \omega& -4 \omega+4\cr
12 \omega+9& 3-12 \omega& 0& -3 \omega& 0& 6 \omega+18& 0& -15 \omega-12& 3 \omega\cr
-3 \omega-6& 6 \omega+3& 0& 0& 0& 3 \omega-3& 0& 3 \omega+6& 0\cr
\end{array}\right]$$
It has rank~$9$ (rows $1,2,3,6,7,8,11,12,13$ are linearly independent). 
Consider the following functions:
$$g_1=45f_4+27f_5+(3-6\omega)f_6-(10\omega+5)f_7-(6\omega+3)f_8$$
$$g_2=-18f_4+(6\omega-6)f_5+6\omega f_6+(4\omega+2)f_7+(2\omega+2)f_8$$
$$g_3=126f_9+(63\omega+126)f_{10}+(105\omega+126)f_{11}+(161\omega+112)f_{12}+(189\omega+42)f_{13}$$
Their differentials at $m_0$ are identically $0$ and the Hessians
$$\left[\begin{matrix}{\partial^2g_k\over\partial u\partial u}&{\partial^2g_k\over\partial u\partial v}\cr
{\partial^2g_k\over\partial v\partial u}&{\partial^2g_k\over\partial v\partial v}\end{matrix}\right],\qquad k=1,2,3$$
at $m_0$ are equal to 
$$
\left[\begin{matrix}-18\omega-18&-30\omega-12\cr
-30\omega-12&-12\omega+54\end{matrix}\right],\quad
\left[\begin{matrix}4\omega+8&16\omega+8\cr
16\omega+8&16\omega-16\end{matrix}\right],\quad
\left[\begin{matrix}-126\omega+42&42\omega+84\cr
42\omega+84&42\omega+42\end{matrix}\right].$$
These three matrices are linearly independent.
\end{Lemma}

\begin{proof}
This is a straightforward calculation and a joy of substitution.
\end{proof}

Now we can finish the proof of the Theorem.
It suffices to show that the scheme $\cZ$ cut out by 
the ideal $\langle f_1,\ldots,f_{13}\rangle$ is zero-dimensional at $m_0$.
This would follow at once if the tangent cone of $\cZ$ at $m_0$ is zero-dimensional.
By the Lemma, the ideal of the tangent cone contains functions $w_i-w_i^0$ (for $i=1,\ldots,9$),
$(u-u^0)^2$, $(u-u^0)(v-v^0)$, and $(v-v^0)^2$, which clearly cut out $m_0$ set-theoretically.
\end{proof}

\begin{Remark}
The dual Hesse configuration is a $q=3$ case of the $\Ceva(q)$ arrangement
with $3q$ lines that satisfy 
$$\Ga_1\ldots\Ga_q=X^q-Y^q,\quad \Ga_{q+1}\ldots\Ga_{2q}=Y^q-Z^q,\quad \Ga_{2q+1}\ldots\Ga_{3q}=Z^q-X^q.$$
We think it is plausible that these hypergraphs also give rise to $1$-dimensional exceptional
loci (on $\oM_{0,q^2+3}$).
\end{Remark}

It is not difficult to compute the numerical class of the exceptional curve, i.e.,~its intersection
numbers with boundary divisors of $\oM_{0,n}$. The following is an immediate
corollary of  Prop.~\ref{blowupdescr}:

\begin{Corollary}
In the setup of Prop.~\ref{propermap}, and assuming $C$ passes only through $p_1,\ldots,p_5$, 
the proper transform $\tilde C\subset\oM_{0,n}$ of $C$ has the following intersections with boundary divisors:
for each line $L_I$, 
$$\delta_I\cdot\tilde C=2-|I\cap\{1,\ldots,5\}|,$$ and for each $k\in I$, 
$$\delta_{I\setminus\{k\}} \cdot\tilde C=\begin{cases}
1& \hbox{if}\ k\le5\cr
0 & \hbox{otherwise}.
\end{cases}$$
Other intersection numbers are trivial.

In particular, in the setup of Th.~\ref{HesseMain}, $C$ has the following numerical class:
$$
\De_{1,b,\ga}+\De_{p, b,\ga}+
\De_{p,2,c,\be}+\De_{2,c,\be}+
\De_{3,a,\al}+\De_{p,3,\al}+
\De_{2,a,\ga}+\De_{n, 2,\ga}+
\De_{n,3,b,\be}$$
$$+\De_{3,b,\be}+
\De_{1,c,\al}+\De_{n,c,\al}+
\De_{1,2,3}+\De_{m, 2,3}+
\De_{m,\al,\be,\ga}+\De_{\al,\be,\ga}+
\De_{a,b,c}+\De_{m,b,c}$$
$$+\De_{1,\be}+2\De_{2,b}+2\De_{2,\al}+
2\De_{3,c}+2\De_{3,\ga}+\De_{a,\be}+
2\De_{b,\al}+2\De_{c,\ga,}$$
where $\De_I$ is a formal curve class that has intersection $1$ with $\de_I$ and
$0$ with the rest of  boundary divisors.
\end{Corollary}

\section{Appendix 
(after Sean Keel \& James McKernan \cite{KM})}

\begin{Theorem}\label{CorolT}
Suppose that the Mori cone $\oNE_1(\oM_{0,n})$ is finitely generated and that
a curve $C \subset \oM_{0,n}$ generates its extremal ray.
If $C\cap M_{0,n}\ne\emptyset$ then $C$ is rigid.
\end{Theorem}

Here we use the following definition:

\begin{Definition}
A proper curve $C$
on a variety $X$ {\em moves} if there is a
proper surjection $p:S \rightarrow B$ from a surface $S$
to a non-singular curve $B$, and a map
$h:\,S \rightarrow X$ with $h(S) \subset X$ a surface, and
a fibre $F$ of $p$ with $h(F)$ set theoretically equal to $C$.
If $C$ does not move then we say that $C$ is {\em rigid}.
\end{Definition}

\begin{Definition}\label{2.1}
We say that an effective Weil
divisor on a projective variety
has ample support if it has the same support as
some effective ample divisor.
\end{Definition}

\begin{Definition}\label{ConvexEdge}
We say that an extremal ray $R$ of a closed convex cone $C \subset \bR^n$
is an {\em edge} if the vectorspace
$R^{\perp} \subset (\bR^n)^*$ (of linear forms that vanish
on $R$) is generated by supporting hyperplanes for $C$. 
This technical condition means that $C$ is ``not rounded'' at $R$.
\end{Definition}

\begin{Definition}\label{AntiNef}
We say that an effective divisor $D$ is antinef if $D.C\leq0$ for any curve $C$ contained in the support of $D$.
\end{Definition}

\begin{Proposition}\label{Prop2.2}
Let $M$ be a ${\Bbb Q}$-factorial projective variety
and $D$ an effective divisor with ample support, each
of whose irreducible components has anti-nef normal
bundle. Let $C \subset M$ be a moving irreducible proper curve
which generates an extremal ray $R$ of the Mori cone such that $R$ is an edge.
Then $R$ is generated by a curve contained in the support of $D$.
\end{Proposition}

\begin{proof}
Let $p:\,S \rightarrow B$
be a proper surjection from a surface $S$
to a non-singular curve $B$ and let
$h:\,S \rightarrow M$ be a morphism such that
$E=h(S) \subset M$ is a surface
and there exists a fibre $F$ of $p$ with $h(F)$ set theoretically equal to $C$.
Clearly, we may assume $S$ is smooth.  
Suppose on the contrary that no
curve in $D \cap E$ generates the same extremal ray as $C$.

Let $D = \sum D_i$ be the decomposition into irreducible components.
Then each $D_i \cap E$ is an effective $\bQ$-Cartier divisor
of $E$, and in particular, is purely one dimensional.
Let $I = D \cap E$, which is non-empty as $D$ has
ample support.

We show next that we can find two irreducible
curves $B_1, B_2 \subset I$
and (after renaming) two divisors $D_1,D_2$ with $B_i \subset D_i$
such that $B_1.D_2>0$ and $B_2.D_1\geq0$.

Choose an irreducible component $B_1$ of $I$ contained
in a maximal number of~$D_i$.
Suppose that $G_1,\dots,G_k$ are the components
of $D$ containing $B_1$.
Since the $D_i$ have anti-nef
normal bundles, $G_i\cdot B_1\le 0$.
Since $D$ has ample support,
there exists a component $D_2$ of $D$ such that
$D_2 \cdot B_1 > 0$.
Let $B_2$ be an irreducible component
of $D_2 \cap I$.  By the choice of $B_1$
there exists $G_i\not\supset B_2$.
We set $D_1= G_i$.

Now choose irreducible curves $B_i' \subset S$ with
$h(B_i') = B_i$. Let $D_i' = h^*(D_i)$. Note that $B_i'$
are multi-sections of $p$ (otherwise $B_i$ generates the
extremal ray). Thus $D_1',D_2'$ have positive intersection
with the general fibre of $p$ and so we can choose
choose $\lambda > 0$ such that
$D_1' - \lambda D_2'$ has zero intersection with the general fibre.

\begin{Claim}\label{PullBack}
Let $L \in N^1(M)$ be a class whose pullback
to $S$ has zero intersection with the general fibre. Then
$h^*(L)$ is pulled back from $N^1(B)$ (numerically).
\end{Claim}

Thus by the claim, $J:=D_1' - \lambda D_2'= h^*(D_1 - \lambda D_2)$
is pulled back from $B$.
Then $J \cdot B_1< 0$ and
$J \cdot B_2 \geq 0$. Since
$J$ is pulled back from $B$, and the $B_i$ are
multi-sections, this is a contradiction.
\end{proof}

\begin{proof}[Proof~Claim \ref{PullBack}]
Since $C$ generates an edge, we may assume $L$ is nef ($L=L_1-L_2$, where $L_1,L_2$ are as in Claim \ref{PullBack} and nef).
Since $L$ is zero on the general fibre of
$p$, it is numerically $p$-trivial. Now apply
\cite[Lem. 4.4]{Keel03}
\end{proof}

\begin{proof}[Proof of Theorem~\ref{CorolT}]
Suppose, on the contrary, that $C$ is a moving curve that intersects $M_{0,n}$ and generates
an edge of the Mori cone.
By \cite[Lem. 3.6]{KM} the boundary has ample support,
and by \cite[Lem. 4.5]{KM} each boundary divisor
has anti-nef normal bundle. Now by Prop.~\ref{Prop2.2}, $C$ is
numerically equivalent to a positive multiple of a curve $C'$ on the boundary. We claim that for 
any $C'$ contained in the boundary, there is some boundary divisor which intersects it negatively.
Then this divisor intersects $C$ negatively and therefore $C$ is contained in the boundary. 
Contradiction.

We prove the claim:
If $C'$ is contained in $\de_{ij}$, consider the Kapranov morphism $\Psi:\M_{0,n}\ra\PP^{n-3}$
with the $i$-th marking as a moving point. Then $\Psi(\de_{ij})= pt$; if we let $\psi_i=\Psi^*\O(1)$, then $C'.\psi_i=0$.
It is not hard to see that the class $\Psi_i$ can be written as $\sum\al_I\de_I$ with $\al_I>0$ for all $I$. If $C'.\de_I\geq0$ for all $I$, then
it follows that $C'.\de_I=0$ for all $I$, which is a contradiction, since the boundary has ample support. If $C'$ is contained in some $\de_I$ 
with $|I|\geq3$, we prove the statement by induction on $|I|$: consider the forgetful map $\pi:\M_{0,n}\ra\M_{0,n-1}$ that forgets a marking $i\in I$.
Then $\pi(\de_I)=\de_{I\setminus\{i\}}$. If $C'$ is not contracted by $\pi$, then by induction, $\pi(C').\de_J<0$ for some $J\subset N\setminus i$.
By the projection formula, $C'.\pi^{-1}\de_J=\pi(C').\de_J<0$ and the statement follows, as $\pi^{-1}\de_J=\de_J+\de_{J\cup\{n\}}$. If $\pi(C')=pt$ 
then  $C'$ is a fiber of $\pi$ and it is an easy calculation to show that $C'.\de_I<0$.
\end{proof}

\end{document}